[1994/06/01]

\documentclass{amsart}
\usepackage{amsthm}
\usepackage{amsmath,amsxtra,amssymb,latexsym, amscd,amsthm}
\usepackage[all]{xy}
\usepackage{mathrsfs}
\usepackage{enumerate}
\usepackage{color}
\usepackage{multicol}
\newtheorem{thm}{Theorem}[section]
\newtheorem{cor}[thm]{Corollary}
\newtheorem{prop}[thm]{Proposition}
\newtheorem{lem}[thm]{Lemma}
\theoremstyle{definition}
\newtheorem{defn}{Definition}[section]
\newtheorem{exam}{Example}[section]

 \theoremstyle{remark}
\newtheorem{rmk}{Remark}[section]
\newtheorem{\theequation}{}[section]

\renewcommand{\theequation}{\thesection.\arabic{equation}}
\usepackage{amsfonts}

\newcommand{\C}{\field{C}}
\newcommand{\R}{\field{R}}

\def\2ovlOB{\overline {\overline \Omega}_B}
\def\1ovlO{\overline \Omega}

\def\R{{\Bbb R}}
\def\BB{{\Bbb B}}
\def\C{{\Bbb C}}
\def\D{{\Bbb D}}

\def\B{{\mathcal B}}

\def\N{{\Bbb N}}

\makeatletter
\@namedef{subjclassname@2020}{%
	\textup{2020} Mathematics Subject Classification}
\makeatother
\begin{document}

\title[WEIGHTED COMPOSITION OPERATORS ON  WEAK HOLOMORPHIC SPACES...]{WEIGHTED COMPOSITION OPERATORS ON  WEAK HOLOMORPHIC SPACES AND APPLICATION TO WEAK BLOCH-TYPE SPACES ON THE UNIT BALL OF A HILBERT SPACE}

\author{Thai Thuan Quang} 
\address{Department of Mathematics and Statistics, 
Quy Nhon University,
170 An Duong Vuong, Quy Nhon, Binh Dinh, Vietnam.}
\email{thaithuanquang@qnu.edu.vn} 


%
%
\date{June 30, 2021.}
\keywords{Operators on Hilbert spaces, Weighted composition operator, Unit ball, Bloch spaces, Boundedness, 	Compactness}

\subjclass[2020]{Primary 47B38, Secondary 30H30, 47B02, 47B33, 47B91}
\maketitle

\bigskip
\begin{abstract} Let $ E $ be a space of  holomorphic functions on the unit ball $ B_X $ of a Banach space $ X.$ In this work, we introduce a Banach structure associated to $ E $ on the linear space  $ WE(Y) $ containing $ Y$-valued holomorphic functions on $ B_X $ such that $ w \circ f \in E $ for every $ w \in W, $ a separating subspace of the dual $ Y' $ of a Banach $ Y. $ We  establish the relation between the boundedness, the (weak) compactness of the weighted composition operators $ W_{\psi,\varphi}: f\mapsto \psi\cdot(f \circ \varphi) $ on   $ E $ and $ \widetilde{W}_{\psi,\varphi}: g\mapsto \psi\cdot(g \circ \varphi) $ on  $ WE(Y)$ via some characterizations of the separating subspace $ W. $ As an application,  via the estimates for the restrictions of   $ \psi $ and $ \varphi $ to  a $ m$-dimensional subspace of $ X $   for some $ m\ge2, $ we characterize the 
	properties mentioned above of  $ W_{\psi,\varphi} $ on  Bloch-type spaces $ \mathcal B_\mu(B_X) $ of   holomorphic functions   on the unit ball $ B_X $ of an infinite-dimensional Hilbert space  as well as  their the associated spaces $W\mathcal B_\mu(B_X,Y), $ where $ \mu $ is a   normal weight on $ B_X. $ 
\end{abstract}

\bigskip
\section{Introduction}
Let $ X, Y $ be complex Banach spaces and $ W\subset Y' $ be a separating subspace of the dual $ Y' $ of $ Y.$ We consider Banach spaces   $ E_1, E_2$    of holomorphic functions  on the unit ball $ B_X $ of   $ X$ and $ W$-associated  spaces $ WE_1(Y), $ $ WE_2(Y) $ of $ Y$-valued holomorphic functions on $ B_X $ 
in the following sense:  
\[ WE_i(Y) := \{f: B_X \to Y:\  f \ \text{is locally bounded and}\ w\circ f \in E_i, \ \forall w \in W\}
\] 
equipped with the norm
\[ \|f\|_{WE_i(Y)} := \sup_{w\in W,\|w\|
	\le 1}\|w\circ f\|_{E_i}, \quad i =1,2.  \]

 For a holomorphic self-map $ \varphi $ of $ B_X $ and a holomorphic function $ \psi $ on $ B_X, $ the weighted composition operators $ W_{\psi,\varphi}: E_1\to E_2 $ and $ \widetilde W_{\psi,\varphi}:  WE_1(Y)\to WE_2(Y) $ are defined by
 \[ \begin{aligned}
  W_{\psi,\varphi}(f) &:= \psi\cdot(f \circ \varphi) \quad  \forall f \in E,\\
  \widetilde{W}_{\psi,\varphi}(g) &:= \psi \cdot (g \circ \varphi) \quad \forall g \in WE_1(Y). 
 \end{aligned} \]
  When the function $ \psi $ is identically $ 1, $ the
operators   reduce to the composition operators $ C_\varphi $ and $ \widetilde C_\varphi. $  A main problem in the investigation
of such operators is to relate function theoretic properties of $ \psi $ and $ \varphi $ to operator theoretic
properties of $ W_{\psi,\varphi}$ and $ \widetilde W_{\psi,\varphi}. $

The problem of studying of weighted composition operators on various Banach spaces of holomorphic functions
on the unit disk or the unit ball (in finite and infinite dimensional spaces), such as Hardy and Bergman spaces, the space $ H^\infty $  of all
bounded holomorphic functions, the disk algebra and weighted Banach spaces with sup-norm, etc. received a special attention of many authors during the past several decades. They  appeared in some works   with different applications. There is a great number of topics on operators of such a type: boundedness and compactness, compact differences, topological structure, dynamical and ergodic properties.

The study of weighted  composition operators between  vector-valued function spaces involves some important basic principles which hold for large classes of function spaces.

In this paper we are interested in the dependence of  the boundedness and the (weak) compactness of $ \widetilde{W}_{\psi,\varphi} $ on the respective properties of 
$W_{\psi, \varphi} $ as well as their restrictions of $W_{\psi, \varphi} $  to some finite-dimensional ones. A positive answer to this problem   is very important because 
in certain situations, the study the above properties of $ \widetilde{W}_{\psi, \varphi}$  can be reduced to studying   $ W_{\psi, \varphi},$ or even, its  restriction  to  the unit ball $ \BB_m \subset\C^m$ for some $ m \in \N. $





First, in Section 2, we introduce some fundamental properties of the    Banach  space of Banach-valued holomorphic functions $ W$-associated to $ E. $ The important result in this section is the linearization theorem for   spaces $ WE(Y)$ (Theorem \ref{thm_1.1}) that allows us to   identify $ WE(Y) $ with the space of $ Y $-valued continuous linear operators on the predual space $ ^*\!E $ of $ E. $ This is usefull in our investigation on the boundedness and the (weak) compactness of operators $ \widetilde{W}_{\psi, \varphi}$ in the next section.

The main theorem in Section 3 is the answer of the  problem posed above. We introduce some assumptions  on the subspace $ W $ of the dual $ Y' $ of $ Y $ to give  necessary as well as  sufficient conditions for the boundednees, the (weak) compactness of   $ \widetilde{W}_{\psi, \varphi} $ via the respective properties of 
 $W_{\psi, \varphi} $ (Theorem \ref{thm_WO}).

The remain part of the paper contains some applications of the above results to the Bloch-type space of Banach-valued holomorphic functions on the unit ball of an infinite-dimensional  Hilbert space.


The notion of  classical Bloch space $ \mathcal B $ of holomorphic functions on  the unit disk $ \BB_1 $ of $ \C $ was extended by R. M. Timoney \cite{Ti1, Ti2} by considering bounded homogeneous domains in $ \C^n, $ for example, the unit ball $ \BB_n $ and the polydisk $ \D_n.  $
Recently, O. Blasco, P. Galindo, A. Miralles \cite{BGM} introduce and investigate Bloch functions on the unit ball $ B_X $ of an infinite-dimensional Hilbert space. It is well known that several results and characterizations of Bloch space on $ \BB_n $ can be extended and still hold in this infinite dimensional setting.

Modifying the well-known definitions and results in \cite{BGM}, in Section 4 we introduce Bloch-type space $ \mathcal B_\mu(B_X, Y) $ of $ Y$-valued holomorphic functions on     the unit ball $ B_X $ of an infinite-dimensional Hilbert space where $ \mu $ is a radial, normal weight on $ B_X. $ It is shown that if the restrictions of a function on $ B_X $ to  finite-dimensional subspaces have their Bloch norms  uniformly bounded then this function belongs to $ \mathcal B(B_X) $  and conversely.   
 Motivated by this fact, we introduce the   notions of gradient norm, radial derivate norm, affine norm on  $ \mathcal B_\mu(B_X, Y)  $  and prove the equivalence between them. Another equivalent norm   for $ \mathcal B(B_X, Y)  $ (in the case $ \mu(z) = 1-\|z\|^2$) which is invariant-modulo the constant functions under the action of the automorphism of the $ B_X $ is also presented in this section.

 In Section 5,  via the estimates for the restrictions of the functions $ \psi $ and $ \varphi $ to $ \BB_m $ for some $ m\ge2, $  we characterize the boundedness and the  compactness  of the operators  $W_{\psi, \varphi}$  between the (little) Bloch-type spaces    $ \mathcal B_\mu(B_X), $ $ \mathcal B_{\mu,0}(B_X)$ as well as   the equivalent relationships between them (Theorems \ref{thm6_1} and \ref{thm_compact}). It should be noted that a necessary condition (but not sufficient) and a sufficient condition (but not necessary) for the compactness of $ W_{\psi, \varphi} $ are also obtained after any necessary minor modifications  for the holomorphic self-map $ \varphi $ (Remark \ref{rmk_6.1} and Theorem \ref{thm_6.4}).  Finally, we finish the paper  with the presentation   the necessary as well as sufficient conditions for the boundedness and the (weak) compactness of $ \widetilde{W}_{\psi,\varphi}$ which are immediate consequences of Theorem \ref{thm_WO}.
 
  Throughout this paper, we use the notions $X \lesssim Y$ and $X \asymp Y$ for non negative quantities $X$ and $Y$ to mean $X \le CY$ and, respectively, $Y/C \le X \le CY$ for some inessential constant $C >0.$

 \section{Weak holomorphic spaces and Linearization Theorem}
Let  $X, Y$ be  complex     Banach spaces. Denote by $ B_X $ the closed unit ball  of $ X $ (we write $ \BB_n $ instead of $ B_{\C^n} $). 

 By $H(B_X, Y)$ we denote the vector  space of $Y$-valued holomorphic functions on $B_X.$ A holomorphic function $f \in H(B_X, Y)$ is called locally bounded holomorphic on $B_X$ if for every $z \in B_X$ there exists a neighbourhood $U_z$ of $0\in X$ such that $f(U_z)$ is bounded. Put
$$H_{LB}(B_X, Y) = \big\{f \in H(B_X, Y): \quad f \; \text{is locally bounded on}\; B_X\big\}.$$

 Suppose that $ E $ is a Banach space of holomorphic functions $ B_X \to \C $ such that
\begin{enumerate}
	\item [(e1)] $ E $ contains the constant functions,
	\item [(e2)] the closed unit ball $ B_E $ is compact in the compact open topology $ \tau_{co} $
	of $ B_X.$
\end{enumerate}
It is easy to check that the properties (e1), (e2) are satisfied by a large number of well-known function
spaces, such as  classical  Hardy, Bergman, BMOA, and Bloch spaces.

Let $ W\subset Y' $ be a separating subspace of the dual $ Y' $ of $ Y.$ We say that the space    
	\[ WE(Y) := \{f: B_X \to Y:\  f \ \text{is locally bounded and}\ w\circ f \in E, \ \forall w \in W\}
\] 
equipped with the norm
\[ \|f\|_{WE(Y)} := \sup_{w\in W,\|w\|
	\le 1}\|w\circ f\|_E.  \]
  is the  Banach space $ W$-associated to $ E $ of $ Y$-valued functions.
%
\begin{prop}\label{Prop2.1} Let $X, Y$ be   complex  Banach spaces and $W \subset Y'$ be a separating subspace. Let $ E $ be a Banach space of holomorphic functions $ B_X \to \C $ satisfying (e1)-(e2) and $ WE(Y) $ be the   Banach space $ W$-associated to $ E. $ Then, the following  assertions  hold:
\begin{enumerate}
	\item [\rm (we1)] $ f\mapsto f \otimes y $ defines a bounded linear operator $ P_y: E \to WE(Y) $ for any $ y \in Y, $ where $ (f\otimes y)(z) = f(z)y $ for $ z \in B_X, $
	\item [\rm (we2)] $ g \mapsto w \circ g $ defines a bounded linear operator $ Q_{w}: WE(Y) \to E $ for any $ w \in W, $
	\item [\rm (we3)] For all $ z \in B_X $ the point evaluations $ \widetilde\delta_z: WE(Y) \to (Y,\sigma(Y,W)), $ where $ \widetilde\delta_z(g) = g(z), $   are continuous.
\end{enumerate}
In the case the hypothesis ``separating'' of $ W $ is replaced by a stronger one that $ W $ is ``almost norming'', we obtain the  assertion (we3') below instead of (we3):
\begin{enumerate}
	\item [\rm (we3')] For all $ z \in B_X $ the point evaluations $ \widetilde\delta_z: WE(Y) \to Y$     are bounded.
\end{enumerate}
\end{prop}
Here, the subspace $ W $ of $ Y' $ is called \textit{almost norming} if 
\[ q_W(x) := \sup_{w \in W, \|w\|\le 1}|w(x)| \]
defines an equivalent norm on $ Y. $
\begin{proof}
	(i) Fix $ y \in Y. $ In fact, for every $ f \in E $ we have $ w\circ(f\otimes  y) = w(y)f. $ Then
	\[\begin{aligned}
		\|P_y(f)\|_{WE(Y)} &= \sup_{\|w\|\le 1}\|w\circ(f\otimes y)\|_E  = \sup_{\|w\|\le 1}\|w(y)f\|_E  \\
		&\le \|w\|\cdot \|y\| \cdot\|f\|_E \le \|y\|\cdot \|f\|_E.
	\end{aligned}   \]
	Thus (we1) holds.
	
	(ii) Fix $ w\in W, $ for every $ g \in WE(Y) $ we have
	\[\begin{aligned}
		\|Q_{w}(g)\|_E &= \|w\circ g\|_E = \|w\|\Big\|\dfrac{w}{\|w\|}\circ g\Big\|_E \\
		&\le \|w\|\sup_{\|u\|
			\le 1}\|u\circ g\|_E = \|w\|\cdot \|g\|_{WE(Y)}.	\end{aligned}  \]
	Thus (we2) is true.  
	
(iii) Fix $ z \in B_X. $ 
	Note first that since $ E $ satisfies (e1)
	and (e2), then the evaluation maps $ \delta_z \in E' $  for $ z \in B_X$ where $  \delta_z(f) = f(z) $ for $ f \in E. $

	It is obvious that  $ w(\widetilde{\delta}_z(g)) = \delta_z(w\circ g)$ for every $g \in WE(Y)$ and   for every $ w\in W. $  
	Let $ V $ be a $ \sigma(Y,W)$-neighbourhood of $ 0 $ in $ Y. $ Without loss of genarality we may assume $ V=\{y \in Y:\ |w(y)| < 1\} $ for some $ w \in W. $ Then $ \widetilde{\delta}_z( \|\delta_z\|^{-1}\|w\|^{-1}B_{WE(Y)}) \subset V, $ where  $ B_{WE(Y)} = \{g \in WE(Y):\ \|g\|_{WE(Y)}< 1\}$ is   the unit ball of $ WE(Y). $ Indeed, for every $g \in B_{WE(Y)}$ we  have
	\[\begin{aligned}
		|w(\widetilde{\delta}_z(\|\delta_z\|^{-1}\|w\|^{-1}g))| &= \|\delta_z\|^{-1}|\delta_z(\|w\|^{-1}w\circ g)| \\
		&\le \|\delta_z\|^{-1}\|\delta_z\| \cdot\big\|\|w^*\|^{-1}w\circ g\big\| \\
		&\le  \sup_{u \in W,\|u\|\le 1}\|u\circ g\|_E =\|g\|_{WE(Y)} < 1.
	\end{aligned}   \]
	Thus, (we3) holds for $ (E, WE(Y)). $
	
	In the case where $ W $ is almost norming, since $ q_W $ defines an equivalent norm, there exists $ C>0 $ such that
	\[\begin{aligned}
\|\widetilde{\delta}_z(g)\| &= \|g(z)\| \le Cq_W(g(z))  = C\sup_{w\in W,\|w\|\le1}|w(g(z))| \\
 &\le C\sup_{w\in W,\|w\|\le1}\|w\circ g\| = C\|g\|_{WE(Y)}\quad \forall g \in WE_1(Y).
	\end{aligned}  \]
The assertion (we3') is proved.
\end{proof}
Finally,  we discus the linearization theorem for   spaces $ WE(Y)$ which is usefull in our investigation on the boundedness and the (weak) compactness of operators $ \widetilde{W}_{\psi, \varphi}$ in the next section.

\begin{thm}[Linearization]\label{thm_1.1} Let $X, Y$ be   complex  Banach spaces and $W \subset Y'$ be a separating subspace. Let $ E $ be a Banach space of holomorphic functions $ B_X \to \C $ satisfying (e1)-(e2). Then there exist a Banach space $^*\!E$ and a mapping $\delta_X \in H(B_X, {^*\!E})$ with the following  universal property:  A function $f \in WE(Y)$ if and only if there is a unique mapping $T_f \in L(^*\!E, Y)$ such that $T_f \circ \delta_X = f.$ This property characterize $^*\!E$ uniquely up to an isometric isomorphism.
	
	Moreover, the mapping
	$$ \Phi: f \in WE(Y) \mapsto T_f  \in L(^*\!E, Y)  $$
	is a topological isomorphism.
\end{thm}
\begin{proof}
	Let us denote by $^*\!E$  the closed subspace of all linear functionals $u \in E'$ such that $u\big|_{B_E}$ is $\tau_{co}$-continuous. 
	By the Ng Theorem \cite[Theorem 1]{Ng}   the evaluation mapping
	$$ J: E \to (^*\!E)' $$
	given by
	$$ (Jf)(u) = u(f) \quad \forall u \in  {^*\!E},$$
	is a topological isomorphism. 
	
	Let $\delta_X: B_X \to \ ^*\!E$ be the evaluation mapping given by 
	$$\delta_X(x) = \delta_x$$
	with  
	$\delta_x(g) := g(x)$ for all $g\in E.$ 
	
	Since
	\begin{equation}\label{eq_1}
		(Jg) \circ \delta_X(x) = \delta_x(g) = g(x)
	\end{equation}
	for all $g \in E,$ $x \in B_X,$    and $J$ is surjective, we can check that the map $\delta_X: B_X \to\ ^*\!E$ is  weak holomorphic, and hence holomorphic because $E$ is Banach. 
	
	Next we prove that 
	\begin{equation}\label{dense}
		\text{\rm span}\{\delta_x: \ x\in B_X\} \ \text{\rm  is a dense subspace of $^*\!E$}.
	\end{equation}
	Indeed, otherwise, by the Hahn-Banach theorem we can find $\eta \in (^*\!E)',$ $\eta \neq 0,$ such that $\eta(\delta_x) = 0$ for every $x \in B_X.$ But, since $J$ is surjective, we have $\eta = Jg$ for some $g \in E.$ Then we get
	$$ g(x) = \delta_x(g) = \eta(\delta_x) = 0 \quad \forall x \in B_X $$ 
	and hence $\eta = 0,$ a contradiction.
	
	Now we show that $^*\!E$ and $\delta_X$ have required universal property.
	
	First, given a locally bounded function $f: B_X \to Y.$ Assume that there exists $T_f \in L(^*\!E, Y)$ such that $T_f \circ \delta_X = f.$ Since $T_f$ is continuous and $ \delta_X $ is holomorphic, it follows that $ u\circ f \in H(B_X) $ for every $ u \in W. $ Since $ W $ is separating, according to   \cite[Lemma 4.2]{QLD} we have $ f \in H_{LB}(B_X,Y). $  Next, it follows from $(^*\!E)' = E$  that $u \circ f \in E$ for each $u \in W,$ and then $f \in WE(Y).$ 
	
	Now we will prove the converse of the statement. Fix $ f \in WE(Y). $
	
	\vskip0.2cm
	(i)  \textit{The case of $Y = \C$:} We define $T_f := Jf.$ It follows from (\ref{eq_1}) that $T_f \circ \delta_X = f.$ From (\ref{dense}) we obtain the uniqueness of $T_f.$
	
	\vskip0.2cm
	(ii) \textit{The case of $Y$ is Banach:}  We define $T_f: ~ ^*\!E \to W'$   by
	\begin{equation}\label{formula_1} (T_fu)(\varphi) = T_{\varphi \circ f}(u) = u(\varphi \circ f) \quad \forall u \in  {^*\!E,}\ \forall \varphi \in W,
	\end{equation}
i.e. $ T_{\varphi\circ f} $ is defined as in the case (i).

	It is easy to check that $ T_f \in L(^*\!E,W') $ and $ \|T_f\| = \|f\|_{WE(Y)},$ hence, $ \Phi $ is a isometric isomorphism.
	
	 Furthermore, 
	\[ (T_f\delta_x)(\varphi) = (\varphi\circ f)(x)\]
	for every $x  \in B_X $ and $ \varphi \in W $ and, therefore, since $ W $ is separating we get $ T_f\delta_x =f(x) \in Y $ for every $ x\in B_X. $ Then, by (\ref{dense})  $ T_f \in L(^*\!E,Y). $ The uniqueness of $ T_f $ follows
	also from the fact (\ref{dense}) that $ \delta_X(B_X) $ generates a dense subspace of $ ^*\!E. $

	Finally, the uniqueness of $ ^*\!E $ up to an isometric isomorphism follows
from the universal property, together with the isometry $ \|T_f\|=\|f\|_{WE(Y)}. $ This
completes the proof.
\end{proof}
\section{The  weighted composition operators}
\setcounter{equation}{0}
Let $ E_1 $ and $ E_2 $ be  Banach spaces of holomorphic functions $ B_X \to \C $ satisfying the conditions (e1) and (e2).  Let $\psi \in H(B_X)$ and  $\varphi \in S(B_X),$  the set of holomorphic self-maps of $B_X.$  Consider the operators $ W_{\psi,\varphi}: E_1 \to E_2 $ and $ \widetilde{W}_{\psi,\varphi}: WE_1(Y) \to WE_2(Y)$ given by
\[ W_{\psi,\varphi}(f) := \psi\cdot(f \circ\varphi), \quad \widetilde{W}_{\psi,\varphi}(g) := \psi\cdot(g \circ\varphi)\quad \forall f\in E_1, \forall g \in WE_1(Y).\]

\begin{thm}\label{thm_WO}
 Let $X, Y$ be   complex  Banach spaces and $W \subset Y'$ be a  subspace. Let $ E_1 $ and $ E_2 $ be Banach spaces of holomorphic functions $ B_X \to \C $ satisfying (e1)-(e2).  Let $\psi \in H(B_X)$ and  $\varphi \in S(B_X).$  
 \begin{enumerate}
 	\item[\rm (1)] If $ W $ is separating then $ W_{\psi,\varphi} $    is bounded if and only if $ \widetilde{W}_{\psi,\varphi}$ is bounded;
 	\item[\rm (2)] If $ W $ is almost norming and   $W_{\psi,\varphi}    $ compact then:
 	\begin{itemize}
 		\item[\rm (a)]  $\widetilde{W}_{\psi,\varphi} $ is  compact if and only if the identity map $ I_Y: Y \to Y $ is   compact,     i.e. $ \dim Y <\infty;$
 			\item[\rm (b)]  $\widetilde{W}_{\psi,\varphi} $ is weakly compact if and only if the identity map $ I_Y: Y \to Y $ is weakly compact.
 	\end{itemize}
 \end{enumerate}
\end{thm}
\begin{proof}
(1) Let $ y\in Y, $   $ w\in W $ such that $ \|y\| = \|w\| = 1 $ and $ w(y) = 1. $ Consider the maps $ P_y $ and $ Q_w $ as in  Proposition \ref{Prop2.1}. It is easy to check that
\[ W_{\psi,\varphi} = Q_w \circ \widetilde{W}_{\psi, \varphi} \circ P_y. \]
By Proposition \ref{Prop2.1}, the operators $P_y, Q_w $ are bounded, hence we have $ W_{\psi,\varphi} $ is bounded if $ \widetilde{W}_{\psi,\varphi} $  is bounded.

In the converse direction, note that for every $ g\in WE_1(Y) $ and $ w \in W $ we have
\[ \|w\circ (\widetilde{W}_{\psi,\varphi}(g))\|_{E_2} = \|W_{\psi,\varphi}(w\circ g)\|_{E_2} \le \|W_{\psi,\varphi}\|\cdot\|w\circ g\|_{E_1}. \]
Consequently, $ \|\widetilde{W}_{\psi,\varphi}\| \le \|W_{\psi,\varphi}\|. $ 

(2) Assume that $ W $ is almost norming and $  W_{\psi,\varphi} $ is compact. Let $ z_0 \in B_X $ such that $ \psi(z_0) \neq 0. $ Put
\[\begin{aligned}
	R: Y&\to WE_1(Y)\\
	y &\mapsto f_y,
\end{aligned}\quad \text{and}\quad
\begin{aligned}
	S: WE_2(Y) &\to Y \\
	h &\mapsto \psi(z_0)^{-1}h(z_0),
\end{aligned} \]
where $ f_y(z) := 1 \otimes y \equiv y $ for every $ z \in B_X. $ Here, the function $ 1 \in E_1 $ by the condition  (e1). It is easy to see that
\begin{equation}\label{eq_3.1}
I_Y = S \circ \widetilde{W}_{\psi, \varphi} \circ R.
\end{equation}  
 Note that $ S = \psi(z_0)^{-1}\widetilde{\delta}_{z_0}. $  It follows from (we1) and (we3') that $ R $ and $ S $ are bounded operators. 
 

Therefore, by (\ref{dense}) we have $ (W_{\psi,\varphi})^*(^*\!E_2) \subset {^*\!E_1}. $ Put $ T_{\psi,\varphi}: L(^*\!E_1, Y) \to L(^*\!E_2, Y) $ given by
\[ A \mapsto I_Y \circ A \circ (W_{\psi,\varphi})^*\big|_{^*\!E_2}. \]
By the compactness of $ W_{\psi,\varphi} $ we have $ (W_{\psi,\varphi})^*\big|_{^*\!E_1}: {^*\!E_2} \to {^*\!E_1} $ is compact. 
Next, we will show that
\begin{equation}\label{eq_3.2}
\widetilde{W}_{\psi,\varphi} = \Phi_2 \circ T_{\psi,\varphi} \circ \Phi_1^{-1}
\end{equation}  
where $ \Phi_1:  L(^*\!E_1, Y) \to WE_1(Y)$ and $ \Phi_2:  L(^*\!E_2, Y) \to WE_2(Y)$ are isometric isomorphisms from Theorem  \ref{thm_1.1}. 
Indeed, we have
\[\begin{aligned}
	\big((\Phi_2 \circ T_{\psi,\varphi} \circ \Phi_1^{-1})(f)\big)(z) &= \Big(\Phi_2(\Phi_1^{-1}(f)\circ(W_{\psi,\varphi})^*\big|_{^*\!E_2})\Big)(z) = \Phi_1^{-1}(f)\big((W_{\psi,\varphi})^*(\delta_z)\big)\\
	&=\Phi_1^{-1}(f)(\psi(z)\delta_{\varphi(z)}) = \psi(z)\Phi_1^{-1}(f)(\delta_{\varphi(z)} = \psi(z)(f\circ \varphi)(z)
\end{aligned} \]
for every $ f \in E_1 $ and $ z \in B_X. $

  (a)  First assume that $ \widetilde{W}_{\psi,\varphi} $ is   compact. By the boundedness of the operators $ R,S, $ (\ref{eq_3.1}) implies the compactness of  $ I_Y. $ 
  
 Now, suppose $ I_Y $ is compact. By the compactness of  $ (W_{\psi,\varphi})^*\big|_{^*\!E_1},  $   it follows from \cite[Theorem 2.1 and Remark 2.4]{ST} that $ T_{\psi,\varphi} $ is  compact. Hence, by (\ref{eq_3.2}), $ \widetilde{W}_{\psi,\varphi} $ is  compact.
  
  (b)  We assume that $ \widetilde{W}_{\psi,\varphi} $ is weakly compact.  Then, (\ref{eq_3.1}) implies that $ I_Y $ is weakly compact.
  
In the contrary case, suppose $ I_Y $ is weakly compact.  
By an argument analogous   to the case (a) but using \cite[Proposition 2.3 and Remark 2.4]{ST} instead of \cite[Theorem 2.1 and Remark 2.4]{ST} we get $ \widetilde{W}_{\psi,\varphi
} $ is weakly compact.
 \end{proof}
\begin{rmk} In the case $ W $ is separating and $ Y $ is separable, we have $ W $ is almost norming and the identity operator $ I_Y $ is weakly compact. Then as in (b) we get $ \widetilde{W}_{\psi,\varphi} $ is weakly compact  if $ W_{\psi,\varphi} $ is compact.  
	\end{rmk}
\section{The Bloch-type spaces on the unit ball of a Hilbert space}
\setcounter{equation}{0}

Throughout the forthcoming, unless otherwise specified, we shall denote by $X$   a complex Hilbert space with the open unit ball $ B_X $ and $ Y $ a Banach space. Denote
\[ H^\infty(B_X, Y) =  \Big\{f \in H(B_X,Y): \ \sup_{z \in B_X}\|f(z)\| <\infty\Big\}. \]

It is easy to check that $ H^\infty(B_X,Y) $ is Banach under the sup-norm
\[ \|f\|_\infty := \sup_{z \in B_X}\|f(z)\|. \]

Let $ (e_k)_{k \in \Gamma}$ be an orthonormal basis of $ X $ that we fix at once. Then every $ z \in X $ can be written as \[ z = \sum_{k \in \Gamma}z_ke_k, \quad  \overline{z} = \sum_{k \in \Gamma}\overline{z_k}e_k.\]

Given $ f \in H(B_X,Y) $ and $ z \in B_X. $ Then, for each $ u\in Y', $ we denote by $ \nabla(u\circ f)(z) $ the gradient of $ u\circ f $ at $ z. $ It is the unique element in $ X $ representing the linear functional $ u\circ f'(z) = (u\circ f)'(z) \in Y'.$ We can write 
\[ \nabla(u\circ f)(z) = \Big(\frac{\partial(u\circ f)}{\partial z_k}(z)\Big)_{k\in \Gamma}, \] 
 hence, for every $ x \in X $
\[ (u \circ f'(z))(x) = (u\circ f)'(z)(x) = \sum_{k\in \Gamma}\frac{\partial(u\circ f)}{\partial z_k}(z)x_k = \langle x, \overline{\nabla(u\circ f)(z)}\rangle. \]
Now for every $ z \in B_X $ we define
\[\begin{aligned}
	 \|\nabla f(z)\| &:= \sup_{u\in Y', \|u\|=1}\|\nabla(u\circ f)(z)\|, \\
	 \|Rf(z)\|&:=  \sup_{u\in Y', \|u\|=1}|R(u\circ f)(z)|,
\end{aligned}\]
where
\[ R(u\circ f)(z) = \langle\nabla(u\circ f)(z), \overline{z}\rangle \]
 is the radial derivative of $ u\circ f $at $ z. $
It is obvious that $ \|Rf(z)\| \le \|\nabla f(z)\|\|z\| <  \|\nabla f(z)\| $ for every $ z \in B_X. $

For $\varphi \in S(B_X),$  the set of  holomorphic self-maps on $ B_X$ we write $ \varphi(z) = \sum_{k\in \Gamma}\varphi_k(z) $ and $ \varphi'(z): X \to X$ its derivative at $ z, $ and $ R\varphi(z) = \langle \varphi'(z),\overline{z}\rangle $ its radial derivative at $ z. $



\begin{defn}\label{weight}
	A positive, continuous function $ \mu $ on the interval $ [0,1) $ is called normal  if there are  three constants $0 \le \delta < 1$ and $0 < a < b < \infty$ such that
	\begin{equation}\label{w_1}
		\frac{\mu(t)}{(1 - t)^{a}} \ \text{is decreasing on $[\delta, 1)$},\quad \lim_{t\to1}\frac{\mu(t)}{(1 - t)^{a}}=0, \tag{$ W_1 $}
	\end{equation}
	\begin{equation}\label{w_2}
		\frac{\mu(t)}{(1 - t)^{b}} \ \text{is increasing on $[\delta, 1)$},\quad \lim_{t\to1}\frac{\mu(t)}{(1 - t)^{b}} =\infty. \tag{$ W_2 $}
	\end{equation}
	If we say that a function $ \mu: B_X \to [0, \infty) $ is normal, we also assume that it is radial, that is, $ \mu(z) = \mu(\|z\|) $ for every $ z\in B_X. $
	
	Then, it follows from ($W_1$) that  a normal function $ \mu $ is strictly decreasing on $ [\delta, 1) $ and $ \mu(t) \to 0 $ as $ t \to 1.$
	
%
%
%
On the other hand, by  (\ref{w_2}) we easily obtain   that
		\begin{equation}\label{S_mu}
S_\mu := \sup_{t \in [0,1)}\frac{(1-t)^b}{\mu(t)} <\infty.
	\end{equation}   
\end{defn}
%

Throughout this paper, the weight $ \mu  $ always is assumed  to be   normal.  In the sequel,  when no confusion can arise,   we will use the symbol  $\lozenge$ to denote either $\nabla$ or $R.$


We define  Bloch-type spaces   on the unit ball $ B_X $ as follows:
 \[
 \mathcal B^\lozenge_\mu(B_X,Y) := \Big\{f \in H(B_X,Y):\ 
 \|f\|_{s\mathcal B_\mu^\lozenge(B_X,Y)}:= \sup_{z \in B_X}\mu(z)\|\lozenge f(z)\| <\infty\Big\}. \]
It is easy to check $ \|\cdot\|_{s\mathcal B^\lozenge_\mu(B_X,Y)}  $  is a semi-norm on $  \mathcal B_\mu^\lozenge(B_X,Y) $  and this space is Banach under the sup-norm
	\[
	 \|f\|_{\mathcal B_\mu^\lozenge(B_X,Y)} := \|f(0)\| + \|f\|_{s\mathcal B_\mu^\lozenge(B_X,Y)}.\]
		We also define  little Bloch-type spaces   on the unit ball $ B_X $ as follows:
		\[
			\mathcal B^{\lozenge}_{\mu,0}(B_X,Y) := \Big\{f \in \mathcal B_\mu^\lozenge(B_X,Y):\ \lim_{\|z\|\to 1}\mu(z)\|\lozenge f(z)\| = 0\Big\} \]
	endowed with the norm induced by $ \mathcal B_\mu^\lozenge(B_X,Y). $ 
		
			In the case $ Y=\C $ we write $ \mathcal{B}_\mu^\lozenge(B_X),$  $ \mathcal{B}_{\mu,0}^\lozenge(B_X)$  instead of the respective notations.
			
			It is clear that for every separating subspace $ W $ of $ Y $ we have
			 \[\mathcal B_\mu^\lozenge(B_X,Y) \subset WB_\mu^\lozenge(B_X)(Y),\quad   \mathcal B^{\lozenge}_{\mu,0}(B_X,Y) \subset W\mathcal B^{\lozenge}_{\mu,0}(B_X)(Y). \] 
%
%
	
	For $ \mu(z) = 1 -\|z\|^2 $ we write $ \mathcal B^\lozenge(B_X, Y) $ instead of $ \mathcal B^\lozenge_\mu(B_X, Y) $ and when $ \dim X = m, $ $ Y =\C $ we obtain correspondingly the classical Bloch-type space $ \mathcal B^\lozenge(\BB_m). $
	
We will show below that the study of  Bloch-type spaces   on the unit ball can be reduced to studying functions defined on finite dimensional subspaces.

For each $ m \in \N $ we denote
\[ z_{[m]} := (z_1, \ldots,z_m) \in \BB_m\]
where $ \BB_m $ is the open unit ball in $ \C^m. $ For $ m \ge 2 $ by
\[ OS_m := \{x = (x_1, \ldots, x_m),\ x_k \in X, \langle x_k, x_j\rangle = \delta_{kj}\} \]
we denote the family of orthonormal systems of order $ m. $
 
It is clear that $ OS_1 $ is the unit sphere of $ X. $

For every $ x \in OS_m, $ $ f \in H(B_X,Y) $ we define
\[ f_x(z_{[m]})  = f\Big(\sum_{k=1}^mz_kx_k\Big).\]
Then  
\[ \nabla(u\circ f_x)(z_{[m]}) = \Big(\dfrac{\partial (u\circ f_x)}{\partial z_j}\Big(\sum_{k=1}^mz_kx_k\Big)\Big)_{j \in \Gamma}\quad \text{for every} \  u \in Y', \]
and hence
\begin{equation}\label{eq_2.1}
\Big\|\nabla f_x(z_{[m]})\Big\| =\Big\|\nabla f\Big(\sum_{k=1}^mz_kx_k\Big)\Big\|.
\end{equation} 
Now, for each finite subset $ F \subset \Gamma, $ in symbol $ |F| <\infty, $ we denote by $ \BB_{[F]} $ the unit ball of $ \text{span}\{e_k,\ k \in F\} $ and  $ f_F = f_x $ where $ x = \{e_k,\ k \in F\}.$  
For each $ z \in B_X $ and each $ F \subset \Gamma $ finite we write
\[ z_F = \sum_{k \in F}z_ke_k \in \BB_{[F]}. \]
\begin{defn} Let $ \BB_1 $ be the open unit ball in $ \C $ and $ f \in H(B_X,Y). $
	We define an affine semi-norm as follows
\[ 	\|f\|_{s\mathcal{B}_\mu^{\rm aff}(B_X,Y)} := \sup_{\|x\|=1} \|f(\cdot\, x)\|_{s \mathcal B_\mu(\BB_1,Y)} 
 \]
where $ f(\cdot\, x): \BB_1 \to Y $ given by $ f(\cdot\, x)(\lambda) = f(\lambda x)$ for every $ \lambda \in \BB_1, $ and
\[ \|f(\cdot\, x)\|_{s \mathcal B_\mu^R(\BB_1,Y)} =\sup_{\lambda \in \BB_1}\mu(\lambda x)\|f'(\cdot\, x)(\lambda)\|. \]

It is easy to see that  $\|\cdot\|_{s\mathcal{B}_\mu^{\rm aff}(B_X,Y)} $ is a semi-norm on $ \mathcal B_\mu(B_X,Y). $ We denote
\[  \mathcal B_\mu^{\rm aff}(B_X,Y) := \{f \in \mathcal B_\mu (B_X,Y):\ \|f\|_{s\mathcal{B}_\mu^{\rm aff}(B_X,Y)}  <\infty\}. 
 \]
It is also easy to check that   $ \mathcal B_\mu^{\rm aff}(B_X,Y) $ is Banach under the  norm
\[\|f\|_{\mathcal B_\mu^{\rm aff}(B_X,Y)}:= \|f(0)\| + \|f\|_{s\mathcal B_\mu^{\rm aff}(B_X,Y)}.
\] 

We also define  little affine Bloch-type spaces   on the unit ball $ B_X $ as follows:
\[ \mathcal B_{\mu,0}^{\rm aff}(B_X,Y) := \big\{f \in \mathcal B_\mu^{\rm aff}(B_X,Y): \lim_{|\lambda| \to 1}\sup_{\|x\|=1}\mu(\lambda x)\|f'(\cdot\, x)(\lambda)\|=0\big\}. \]

As the  above, for $ \mu(z) = 1 - \|z\|^2 $ we use notation $ \mathcal B $ instead of $ \mathcal B_\mu. $
 \end{defn}

\begin{prop}\label{prop_OSm}
	Let $ f \in H(B_X,Y). $ The following are equivalent:
	\begin{enumerate}[\rm(1)]
		\item $ f \in \mathcal{B}_\mu^\nabla(B_X,Y); $
		\item $ \sup\{\|f_F\|_{\mathcal B_\mu^\nabla(\BB_{[F]},Y)}: \ F \subset \Gamma, |F| < \infty\}<\infty;$
\item $ \sup_{x \in OS_m}\|f_x\|_{\mathcal{B}_\mu^\nabla(\BB_m,Y)} < \infty  $ \ for every $ m \ge 2; $
\item There exists $ m \ge 2 $ such that $ \sup_{x \in OS_m}\|f_x\|_{\mathcal{B}_\mu^\nabla(\BB_m,Y)} < \infty.  $
	\end{enumerate}
Moreover, for each $ m\ge 2 $
\begin{equation}\label{eq_Prop4.1}
\|f\|_{s\mathcal{B}_\mu^\nabla(B_X,Y)} = \sup_{|F|<\infty}\|f_F\|_{s\mathcal{B}_\mu^\nabla(\BB_{[F]},Y)}= \sup_{x\in OS_m}\|f_x\|_{s\mathcal{B}_\mu^\nabla(\BB_m,Y)}. 
\end{equation}  
\end{prop}
\begin{proof}
	(1) $ \Rightarrow $ (2):  Let $ F  \subset \Gamma, $ $ |F| <\infty $   and $ z_{[F]} \in \BB_{[F]}. $ According to (\ref{eq_2.1}) 
	\[ \big\|\nabla f_F\big(z_{[F]}\big)\big\|  = \Big\|\nabla f\Big(\sum_{j \in F}z_je_j\Big)\Big\|. \]
	Denote $ \mu^{[F]} = \mu\big|_{\text{span}\{e_k, k \in F\}}. $ Since $ \big\|\sum_{j \in F}z_je_j\big\| = \big\|z_{[F]}\big\| $ we get 
	\begin{equation}\label{eq_Prop4.1a}\begin{aligned}
	\|f_F\|_{s\mathcal{B}_\mu^\nabla(\BB_{[F]},Y)}&= \sup_{z_{[F]} \in \BB_{[F]}}\mu^{[F]}(z_{[F]})\|\nabla f_F(z_{[F]})\| \\
	&\le \sup_{z \in B_X}\mu^{[F]}(z_{[F]})\Big\|\nabla f\Big(\sum_{j \in F}z_je_j\Big)\Big\|\\
	&\le \|f\|_{s\mathcal{B}_\mu^\nabla(B_X,Y)}.
	\end{aligned} 
\end{equation}
In particular, we obtain (2).

	(2) $ \Rightarrow $ (1):  Let $ z = \sum_{k\in \Gamma}z_ke_k. $ We denote the partial sums of this series by $ s_n. $ Because $ f $ is holomorphic, $ \frac{\partial f}{\partial z_j} $ are continuous. Then
	\[\begin{aligned}
\|\nabla f(z)\| &= \sup_{u\in Y', \|u\|=1}\|\nabla(u\circ f)(z)\| \\
&= \sup_{u\in Y', \|u\|=1}\lim_{n\to\infty}\|\nabla(u\circ f)(s_n)\| \\
&\le  \sup_{u\in Y', \|u\|=1}\sup_{F \subset \Gamma, |F|<\infty}\|\nabla(u\circ f_F)(z_{[F]})\|\\
&=\sup_{F \subset \Gamma, |F|<\infty}\|\nabla  f_F(z_{[F]})\|.
	\end{aligned}\]
	Then, it follows from the assumption (2) and $   \|z_{[F]}\| \le \|z\|,$  that
	\begin{equation}\label{eq_Prop4.1b}\begin{aligned}
		\mu^{[F]}(z_{[F]})\|\nabla f(z)\| &\le \mu^{[F]}(z_{[F]})\|\nabla f(z)\| \\
		&\le \sup_{F\subset \Gamma, |F| <\infty} \mu^{[F]}(z_{[F]})\|\nabla f_{F}(z_{[F]})\|  <\infty.
	\end{aligned}\end{equation}
Thus $ f \in \mathcal{B}_\mu^\nabla(B_X,Y). $

	(1) $ \Rightarrow $ (3): It is analogous to 	(1) $ \Rightarrow $ (2).
	
		(3) $ \Rightarrow $ (4): It is obvious.
		
			(4) $ \Rightarrow $ (1): Assume that there exists $ m \ge 2 $ such that $ \sup_{x\in OS_m}\|f_x\|_{\mathcal{B}(B_X,Y)}<\infty. $ We fix $ z\in B_X, $ $ z \neq 0. $ Consider $ x =(\frac{z}{\|z\|},x_2, \ldots, x_m) \in OS_m $ and put $ z_{[m]} := (\|z\|, 0,\ldots,0) \in \BB_m. $ Then $ \|z_{[m]}\| = \|z\| $ and
			\begin{equation}\label{norm_z_x} \big\|\nabla f_x(z_{[m]})\big\| = \Big\|\nabla f\Big(\sum_{k=1}^mz_kx_k\Big)\Big\| = \|\nabla f(z)\|. \end{equation}
		This implies that
		\begin{equation}\label{eq_Prop4.1c}\begin{aligned}
\|f\|_{s\mathcal{B}_\mu^\nabla(B_X,Y)} &= \sup_{z \in \B_X}\mu(z)\|\nabla f(z)\| \\
&\le \sup_{z \in \B_X}\mu(z_{[m]})\|\nabla f_x(z_{[m]})\| \\
&\le  \sup_{x \in OS_m}\|f_x\|_{\mathcal{B}(\BB_m,Y)} < \infty.
		\end{aligned}\end{equation} 
	Thus $ f \in \mathcal{B}_\mu^\nabla(B_X,Y). $
	
	On the other hand, it is obvious that  
	\begin{equation}\label{eq_Prop4.1d}
 \sup_{x \in OS_m}\|f_x\|_{\mathcal{B}_\mu^\nabla(\BB_m,Y)} \le \|f\|_{s\mathcal{B}_\mu^\nabla(B_X,Y)} \quad \forall m\ge 2.
	\end{equation}
Hence, we obtain (\ref{eq_Prop4.1}) from (\ref{eq_Prop4.1a}), (\ref{eq_Prop4.1b}), (\ref{eq_Prop4.1c}) and (\ref{eq_Prop4.1d}).
\end{proof}

\begin{rmk} The proposition is not true for the case $ m=1. $ 
		Indeed, let $ X $ be a Hilbert space with the orthonormal basis $ \{e_n\}_{n\ge1}. $ Consider   $ f: B_X \to \C $ given by
	\[ f(z) := \sum_{n=1}^\infty\frac{\langle z, e_n\rangle}{\sqrt{n}} \quad \forall z \in B_X. \]
 Then $ f $ is holomorphic on $ B_X $ because
 \[ \sum_{n=1}^\infty\dfrac{|\langle z, e_n\rangle|^2}{n} \le \sum_{n=1}^\infty|\langle z, e_n\rangle|^2 = \|z\|^2 < 1. \]
 For each $ x = \sum_{n=1}^\infty\langle x, e_n\rangle e_n \in OS_1 $ and for every $ z_{[1]} := z_1 \in \BB_1 $ we have
 \[ \nabla f_x(z_{[1]}) = \nabla f(z_1x_1) = \nabla f\Big(\sum_{n=1}^\infty\frac{\langle z_1x_1, e_n\rangle}{\sqrt{n}}\Big), \]
 and thus, since
 \[ \|\nabla f_x(z_{[1]})\|^2= |x_1|^2 \le 1\] 
 we get 
 \[ \sup_{x \in OS_1}\|f_x(z_{[1]})\|_{\mathcal B_\nabla(\BB_1)} = \sup_{x \in OS_1}(1-\|z_{[1]}\|^2) \|\nabla f_x(z_{[1]})\| \le1. \]
 
 However, $ f \not\in \mathcal{B}_\nabla(B_X)$ because for every $ z \in B_X $ we have
 \[ \|\nabla f(z)\|^2 = \sum_{n=1}^\infty\Big|\frac{\partial f}{\partial z_n}(z)\Big|^2  = \sum_{n=1}^\infty \frac{1}{n}. \]
\end{rmk}

\begin{prop}\label{prop_OSm_0}
	Let $ f \in H(B_X,Y). $ The following are equivalent:
	\begin{enumerate}[\rm(1)]
		\item $ f \in \mathcal{B}_{\mu,0}^\nabla(B_X,Y); $
		\item $\forall \varepsilon>0\ \exists \varrho>0\ \forall z \in B_X\ \text{with}\ \|z_{[F]}\|>\varrho$ for every $ F\subset \Gamma, |F|<\infty $
		 $$ \sup_{F\subset\Gamma, |F|<\infty}\mu(z_{[F]})\|\nabla f_F(z_{[F]})\|<\varepsilon;$$
		\item  $\forall \varepsilon>0\ \exists \varrho>0\ \forall z \in B_X\ \text{with}\ \|z_{[m]}\|>\varrho$ for every $ m\ge2 $
		$$\sup_{m\ge2} \sup_{x \in OS_m}\mu(z_{[m]})\|\nabla f_x(z_{[m]})\|<\varepsilon; $$
		\item  $\exists  m \ge 2 \ \forall \varepsilon>0\ \exists \varrho>0\ \forall z \in B_X\ \text{with}\ \|z_{[m]}\|>\varrho$
		$$\sup_{x \in OS_m}\mu(z_{[m]})\|\nabla f_x(z_{[m]})\|<\varepsilon.$$
	\end{enumerate}
\end{prop}
\begin{proof}
	The implications (1) $\Rightarrow  $ (2) $\Rightarrow  $ (3) $\Rightarrow  $ (4) are obvious. 
	
	(4) $ \Rightarrow $ (1):  The proof is straight-forward by putting $ x \in OS_m $ and $ z_{[m]} \in \BB_m $ as in the proof of (4) $\Rightarrow  $ (1) in Proposition \ref{prop_OSm} for each $ z\in B_X $ with $ \|z\| >\varrho.$ 
\end{proof}
In the next proofs below we need the following lemma.
\begin{lem}\label{lem_3.2}
For every $ f \in \mathcal{B}_\mu^\nabla(B_X,Y) $ and $ x \in X $ with $ \|x\|=1 $ we have
\begin{equation}\label{eq_2.2}
	Rf(\lambda x) = \lambda f'(\cdot\, x)(\lambda)\quad \forall \lambda \in \BB_1
	\end{equation}
and
\begin{equation}\label{eq_2.3}
	f'(\cdot\, x)(\lambda)(\mu) = f'(\lambda x)(\mu x)\quad \forall \lambda, \mu \in \BB_1.
\end{equation}
\end{lem}
\begin{proof}
(i) First, it follows from the Bessel inequality that every $ x \in X $ has only a countable number of non-zero Fourier
coefficients $ \langle x, e_j\rangle. $ Indeed, for every $\varepsilon > 0 $ the set $ \{j \in \Gamma: \ |\langle x, e_j\rangle| > \varepsilon \} $ is finite. Then we still have $ x = \sum_{j \in \Gamma}\langle x, e_j\rangle e_j  =  \sum_{j \in \Gamma}x_j e_j  $ where the sum is in fact a countable one, and it is independent of the particular enumeration of the countable number of non-zero summands. Hence, we can write $ x = \sum_{j=1}^\infty x_je_j. $	
Then, by the definitions of $ f(\cdot\, x) $  and $ f'(\cdot\, x)(\lambda) $ we have
\[ \begin{aligned}
	&\Bigg\|\dfrac{1}{t}\sum_{k=1}^\infty\Bigg(f\Big(\sum\limits_{j=1}^k\lambda x_je_j + t\lambda x_ke_k + \sum\limits_{j=k+1}^\infty(\lambda+t\lambda)x_je_j\Big) -\\
	&\hskip2cm -f\Big(\sum\limits_{j=1}^k\lambda x_je_j  + \sum\limits_{j=k+1}^\infty(\lambda +t\lambda)x_je_j\Big)\Bigg)- \lambda f'(\cdot\, x)(\lambda)\Bigg\| =\\
	&=\Bigg\|\dfrac{f((\lambda+t\lambda)x) - f(\lambda x)}{h} - \lambda f'(\cdot\, x)(\lambda)\Bigg\| = \\
	&=\Bigg\|\dfrac{f(\cdot\, x)(\lambda +t\lambda) - f(\cdot\, x)(\lambda)}{t} - \lambda f'(\cdot\, x)(\lambda)\Bigg\| \to 0 \quad \text{as}\ t \to 0.
\end{aligned}
\]
Hence (\ref{eq_2.2}) is proved.

(ii) For $ \lambda, \eta \in \BB_1$ we have
\[ \begin{aligned}
	 \|\eta &f'(\cdot\, x)(\lambda) -f'(\lambda x)(\eta x)\|  =\\
	&= \Bigg\|\dfrac{f(\cdot\, x)(\lambda +t\eta) - f(\cdot\, x)(\lambda)}{t}-\eta f'(\cdot\, x)(\lambda)   -\dfrac{f(\lambda x + t\eta x) - f(\lambda x) }{t} + f'(\lambda x)(\eta x) \Bigg\| \\
	&\le  \Bigg\|\dfrac{f(\cdot\, x)(\lambda +t\eta) - f(\cdot\, x)(\lambda)}{t}-\eta f'(\cdot\, x)(\lambda)\Bigg\| + \Bigg\|\dfrac{f(\lambda x + t\eta x) - f(\lambda x) }{t} + f'(\lambda x)(\eta x) \Bigg\|\\
	&\to 0 \quad\text{as}\ t \to 0.
\end{aligned} \]
Then $ 	f'(\cdot\, x)(\lambda)(\eta) = \eta f'(\cdot\, x)(\lambda)=  f'(\lambda x)(\eta x), $ and (\ref{eq_2.3}) is proved.
\end{proof}
\begin{prop}\label{prop_4.3}
\begin{enumerate}[\rm(1)]
	\item The spaces 	$ \mathcal{B}_\mu^R(B_X,Y)$  and  $ \mathcal{B}_\mu^{\rm aff}(B_X,Y) $ coincide. Moreover,  \[ \|f\|_{s\mathcal{B}_\mu^R(B_X,Y)} \le \|f\|_{s\mathcal{B}_\mu^{\rm aff}(B_X,Y)}\lesssim \|f\|_{s\mathcal{B}_\mu^R(B_X,Y)}\quad \forall f \in \mathcal{B}_\mu^R(B_X,Y). \]
	\item The spaces 	$ \mathcal{B}_{\mu,0}^R(B_X,Y)$  and  $ \mathcal{B}_{\mu,0}^{\rm aff}(B_X,Y) $ coincide.
\end{enumerate}
\end{prop}
\begin{proof}
(1)(i) Let $ f \in \mathcal{B}_\mu^{\rm aff}(B_X,Y). $ In order to prove $ f \in \mathcal{B}_R(B_X,Y) $ it suffices to show   that
\begin{equation}\label{eq_2.4}
	Rf(z) = \|z\|f'\Big(\cdot \dfrac{z}{\|z\|}\Big)(\|z\|) \quad \forall z \in B_X \setminus \{0\}.
\end{equation}
It is easy to see that (\ref{eq_2.4})  follows immediately  from (\ref{eq_2.2}) for $ y = \frac{z}{\|z\|} $ and $ \lambda = \|z\| $ for every $ z \in B_X \setminus \{0\}. $  Moreover, it follows from (\ref{eq_2.4}) that
\[ \|f\|_{s\mathcal{B}_\mu^R(B_X,Y)} \le \|f\|_{s\mathcal{B}_\mu^{\rm aff}(B_X,Y)}. \]

(ii) Let $ f \in \mathcal B_\mu^R(B_X,Y) $ and $ x \in X $ be such that $ \|x\|=1. $ Since $ f $ is holomorphic at $ 0 \in B_X, $  its derivative  $ f': B_X \to L(X,Y) $ is also holomorphic, and thus there are $ r \in (0,1) $ and $ M > 0 $ such that
\[ \|f'(z)\|_{L(X,Y)} \le M\quad \forall z \in \overline{B}(0,r) := \{u \in X:\ \|u\|\le r\}. \]
Then, by (\ref{eq_2.3}) we have
\[\begin{aligned}
\sup_{|\lambda|\le r}\mu(\lambda x)\|f'(\cdot\, x)(\lambda)\| &=  \sup_{|\lambda|\le r}\mu(\lambda x)\sup_{|\eta|\le1}\|f'(\cdot\, x)(\lambda)(\eta)\| \\
&= \sup_{|\lambda|\le r}\mu(\lambda x)\sup_{|\eta\le1}\|f'(\lambda x)(\eta x)\| \\
&\le \sup_{|\lambda|\le r}\mu(\lambda x)\|f'(\lambda x)\| \le M. 
\end{aligned} \]
For the case where $ \|z\|>r, $ by (\ref{eq_2.2}), (\ref{eq_2.3}) and  the monotony increases of the function $ \frac{1-t}{t} $   we have
\begin{equation}\label{eq_2.4a}
\begin{aligned}
\mu(\lambda x)|f'(\cdot\, x)(\lambda)| &= \mu(\lambda x)|(1-|\lambda|)|f'(\cdot\, x)(\lambda)| +\mu(\lambda x)||\lambda||f'(\cdot\, x)(\lambda)| \\
&\le \mu(\lambda x)|(1-|\lambda|)\dfrac{1-r}{r}|f'(\cdot\, x)(\lambda)| +\mu(\lambda x)|\|Rf(\lambda x)\| \\
&\le \Big(\mu(\lambda x)\dfrac{1-r}{r} +\mu(\lambda x)\Big)\|Rf(\lambda x)\|.
\end{aligned} 
\end{equation}  
This implies that
\[ \sup_{|\lambda| >r}\mu(\lambda x)|f'(\cdot\, x)(\lambda)| \le \dfrac{1}{r}\sup_{z\in B_X}\mu(z)\|Rf(z)\|. \] 
Therefore, $ f \in \mathcal{B}_\mu^{\rm aff}(B_X,Y),$ and we also obtain $ \|f\|_{s\mathcal{B_\mu}^{\rm aff}(B_X,Y)} \le \frac1r\|f\|_{s\mathcal{B}_\mu^R(B_X,Y)}. $

(2) Let $ f \in \mathcal{B}_{\mu,0}^{\rm aff}(B_X,Y). $ Then,  using (\ref{eq_2.4}) it is easy to see that $ f \in \mathcal{B}_{\mu,0}^R(B_X,Y). $ In the converse direction, it follows from (\ref{eq_2.4a}) that $ f \in \mathcal{B}_{\mu,0}^{\rm aff}(B_X,Y)$ if $ f \in \mathcal{B}_{\mu,0}^R(B_X,Y). $
\end{proof}

Next, we will compare the spaces $ \mathcal{B}_\mu^\nabla(B_X,Y) $ and $ \mathcal{B}_\mu^R(B_X,Y). $

We need a vector-valued version of Lemma 4.11 in \cite{Ti1}. First we note that
\begin{equation}\label{eq_2.5} f \in \mathcal{B}_\mu(\BB_1,Y) \quad\text{if and only if}\quad u \circ f \in \mathcal{B}_\mu(\BB_1) \  \text{for all} \ u  \in Y'
\end{equation}
and, interchanging the suprema, we have that
\begin{equation}\label{eq_2.6}
	\|f\|_{\mathcal{B}_\mu^\nabla(\BB_1,Y)} \asymp \sup_{\|u \|=1}\|u \circ f\|_{\mathcal{B}_\mu^\nabla(\BB_1)}.
\end{equation}
\begin{lem}\label{lem_Ti_4.11}
	Let $ f \in \mathcal{B}^{\rm aff}(\BB_2, Y). $ If there exists $ M>0 $ such that for any $ x=(x_1,x_2) \in \BB_2, $ the function $ f(\cdot\, x)(\lambda) $  satisfies $ \|f(\cdot\, x)\|_{s\mathcal{B}^{\rm aff}(\BB_1,Y)}  \le M,$ then
	\begin{equation}\label{eq_2.7}
		\mu((x_1,0))\|\nabla f(x_1,0)\| \le 2\sqrt{2}MR_\mu   \quad \forall x_1 \in \C, |x_1|<1
		\end{equation}  
		where $ R_\mu := 1 + \max_{t\in [0,\delta]}\mu(t)\int_0^\delta\frac{dt}{\mu(t)}. $
\end{lem}
\begin{proof}
	Fix $ u  \in Y' $ with $ \|u\| = 1. $ By the hypothesis,  $ f(\cdot\, x) \in \mathcal{B}(\BB_1,Y). $ Then it follows from (\ref{eq_2.5}) that $ u \circ f(\cdot\, x) \in \mathcal{B}(\BB_1). $ 
	\[ \| u \circ f(\cdot\, x)\|_{s\mathcal{B}^\nabla} \le \|u\| \|f(\cdot\, x)\|_{s\mathcal{B}^{\rm aff}}\le M. \]
	First of all, the hypotheses imply that
	\[ \mu((x_1,0))\Big|\dfrac{\partial(u\circ f)}{\partial x_1}(x_1,0)\Big|\le M. \]
	and so it is sufficient to show that
	\[ \mu((x_1,0))\Big|\dfrac{\partial(u\circ f)}{\partial x_2}(x_1,0)\Big|\le   2\sqrt{2}M. \]
	Indeed, from the hypotheses, we have
	\[ |f(z)-f(0)| = \Big|\int_0^1 \langle \nabla f(tz),\overline{z}\rangle dt\Big|  \le M\int_0^{\|z\|}\dfrac{dt}{\mu(t)}.\]
	Then, using the Cauchy integral formula and a simple estimate, we obtain
	\[\begin{aligned}
		\mu((x_1,0))&\Big|\dfrac{\partial(u\circ f)}{\partial x_2}(x_1,0)\Big| \\
		&\le \mu((x_1,0))\dfrac{1}{2\pi}\int_{|w|=1/\sqrt[4]{2}}\dfrac{\|u\| |f(x_1,w) - f(0) + f(0)- f(x_1,0)|}{|w|^2}dw\\&\le \mu((x_1,0)) \dfrac{2M\int_0^{|x_1|}\frac{dt}{\mu(t)}}{2\pi }\int_{|w|=1/\sqrt[4]{2}}\dfrac{dw}{w^2} \le 2\sqrt{2}MR_\mu
	\end{aligned} \]
	as required.
\end{proof}

\begin{thm}\label{thm_4.3}
	\begin{enumerate}[\rm(1)]
		\item The spaces $ \mathcal{B}_\mu^\nabla(B_X,Y) $ and  $ \mathcal{B}_\mu^R(B_X,Y) $   coincide. Moreover,
		\[ \|f\|_{\mathcal B_\mu^R(B_X,Y)} \asymp \|f\|_{\mathcal B_\mu^\nabla(B_X,Y)}.  \]
			\item The spaces $ \mathcal{B}_{\mu,0}^\nabla(B_X,Y)$ and  $ \mathcal{B}_{\mu,0}^R(B_X,Y) $   coincide.
	\end{enumerate}
 \end{thm}
\begin{proof} (1) Let us show that $ \|f\|_{s\mathcal B^\nabla_\mu(B_X,Y)} \le 2\sqrt{2} R_\mu \|f\|_{s\mathcal B^{\rm aff}_\mu(B_X,Y)} $ and the result follows using  Prposition \ref{prop_4.3}.
	
	Fix $ u  \in Y' $ with $ \|u\| = 1.$  
	Let $ z \in B_X $ and $ v \in X $ with $ \|v\|= 1 $ be fixed. We may assume that $ \dim X \ge 2. $ 
	Then there exist orthonormal unit vectors $ e_1,e_2 \in X $ and $ s, t_1, t_2 \in \C $ with $ |s|<1 $ and $ |t_1|^2+|t_2|^2=1 $ such that $ z = se_1, $ $ v = t_1e_1+t_2e_2. $ For $ f \in \mathcal{B}_\mu^R(B_X,Y) $ put
	\[ F(z_1, z_2) = (u\circ f)(z_1e_1+z_2e_2), \quad (z_2,z_2) \in \BB_2. \]
	Then $ F \in H(B_X) $ and it is easy to check that $ F $ satisfies the assumptions of Lemma \ref{lem_Ti_4.11}. Then
 \[ \mu(z)|\nabla(u\circ f)(z)| = \mu(s)|\nabla(u \circ f)(se_1)| = \mu(s,0)|\nabla F(s,0)| \le 2\sqrt{2}MR_\mu,\]
 hence, $ \|f\|_{s\mathcal B^\nabla_\mu(B_X,Y)} \le 2\sqrt{2}R_\mu \|f\|_{s\mathcal B^{\rm aff}_\mu(B_X,Y)}$ as required.
 
 (2) Because $ \|Rf(z)\| < \|\nabla f(z)\|$ for every $ z \in B_X, $ it suffices to show that $ \mathcal B^R_{\mu,0}(B_X,Y) \subset \mathcal B^\nabla_{\mu,0}(B_X,Y). $ Let $ f \in \mathcal B^R_{\mu,0}(B_X,Y) $ 
 and consider the function $ F(z_1,z_2) $ defined in the proof of the part (1). It is clear that 
  $ RF(z_1,z_2) = R(u\circ f)(z_1e_2+z_2e_2), $ $ \frac{\partial F}{\partial z_1}(z_1,0) = \langle\nabla(u\circ f)(z_1e_1), e_1\rangle $ and  $ \frac{\partial F}{\partial z_2}(z_1,0) = \langle\nabla(u\circ f)(z_1e_1), e_2\rangle.$
  
  Fix $ r_0 \in (\delta,1)$ and assume that $ |z_1| =: r\ge r_0.  $ Since $ \delta <r_0< \sqrt{t^2 +r_0^2(1-(t/r)^2)} \le r $ for $ t \in [0,r] $ and $ \mu $ is strictly decreasing on $ [\delta, 1), $ we have
  \[ \mu(\sqrt{t^2 +r_0^2(1-(t/r)^2)}) \ge \mu(r), \quad  t \in [0,r]. \]
  Then, for $ t \in [0,r] $ by Cauchy's integral formula, we have
  \[ \begin{aligned}
  	\bigg|\frac{\partial(RF)}{\partial z_2}(t,0)\bigg|&= \bigg|\frac{1}{2\pi i}\int_{|z_2| = r_0\sqrt{1-(t/r)^2}}\frac{RF(t,z_2)dz_2}{z_2^2}\bigg| \\
  	&= \bigg|\frac{1}{2\pi}\int_{|z_2| = r_0\sqrt{1-(t/r)^2}}\frac{Rf(te_1+z_2e_2)dz_2}{z_2^2}\bigg| \\
  	&\le \frac{\max_{|z_2| = r_0\sqrt{1-(t/r)^2}}|Rf(te_1+z_2e_2)|}{r_0\sqrt{1-(t/r)^2}} \\
  	&\le \frac{\sup_{r_0\le\|z\|<1}\mu(z)|Rf(z)|}{\mu(r)r_0\sqrt{1-(t/r)^2}}.
  \end{aligned} \]
  Then, by this estimate and \cite[Lemma 6.4.5]{Ru}, for $ |z_1|=r\ge r_0 $ we have
  \[ \begin{aligned}
  	|z_1|\bigg|\frac{\partial F}{\partial z_2}(z_1,0)\bigg|&=\bigg|\int_0^r\frac{\partial(RF)}{\partial z_2}(t,0)dt\bigg| \\
  	&\le \int_0^r\frac{\sup_{r_0\le\|z\|<1}\mu(z)|Rf(z)|}{\mu(r)r_0\sqrt{1-(t/r)^2}} \\
  	&=\frac{\sup_{r_0\le\|z\|<1}\mu(z)|Rf(z)|}{\mu(r)r_0} \int_0^r\frac{dt}{\sqrt{1-(t/r)^2}} \\
  	&= \frac{\pi |z_1|}{2\mu(|z_1|)r_0}\sup_{r_0\le\|z\|<1}\mu(z)|Rf(z)|.
  \end{aligned} \]
It implies that  
\begin{equation}\label{est_B_mu01}
	\bigg|\frac{\partial F}{\partial z_2}(z_1,0)\bigg| \le  \frac{\pi |z_1|}{2\mu(|z_1|)\delta}\sup_{r_0\le\|z\|<1}\mu(z)|Rf(z)| \quad \text{for}\  |z_1|\ge r_0.
\end{equation}

On the other hand, we also have
\begin{equation}\label{est_B_mu02}
	\bigg|\frac{\partial F}{\partial z_1}(z_1,0)\bigg| =  \bigg|\frac{Rf(z_1e_1)}{z_1}\bigg| \le  \frac{1}{\mu(|z_1|)\delta}\sup_{r_0\le\|z\|<1}\mu(z)|Rf(z)| \quad \text{for}\  |z_1|\ge r_0.
\end{equation}
From (\ref{est_B_mu01}) and  (\ref{est_B_mu02})  we obtain
\begin{equation}\label{est_B_mu03} 
	\begin{aligned}
	\mu(z)|\langle\nabla f(z), v\rangle|&= \mu(s)|\langle\nabla f(se_1),t_1e_1 + t_2e_2\rangle| \\
	&= \mu(s)\bigg|t_1\frac{\partial F}{\partial z_1}(s,0) + t_2\frac{\partial F}{\partial z_2}(s,0) \bigg| \\
	&\le \mu(s)\bigg(\bigg|\frac{\partial F}{\partial z_1}(s,0)\bigg|^2 + \bigg|\frac{\partial F}{\partial z_2}(s,0)\bigg|^2\bigg)^{1/2} \\
	&\le \frac{\pi}{\sqrt{2}\delta}\sup_{r_0\le\|z\|<1}\mu(z)|Rf(z)|,\quad \|z\| \ge r_0, \|v\|=1.
\end{aligned} 
\end{equation}
Now, by the hypothesis, for every $ \varepsilon>0 $ we can find $ r_0 \in (\delta,1) $ such that $ \mu(z)\|Rf(z)\| < \varepsilon $ for $ \|z\| >r_0. $ Therefore, it follows from (\ref{est_B_mu03}) that $ \lim_{\|z\|\to1}\mu(z)\|\nabla f(z)\| =0, $ that means $ f \in  \mathcal B^\nabla_{\mu,0}(B_X,Y).  $
 \end{proof}

We can now combine the results of Proposition \ref{prop_4.3} and Lemma  \ref{lem_Ti_4.11} with   an argument analogous  to the Theorem 2.6 in \cite{BGM} and obtain the following theorem:
\begin{thm}\label{thm_main}
 The spaces $ \mathcal{B}_\mu^\nabla(B_X,Y), $ $ \mathcal{B}_\mu^R(B_X,Y) $ and $ \mathcal{B}_\mu^{\rm aff}(B_X,Y)$ coincide. The spaces $ \mathcal{B}_{\mu,0}^\nabla(B_X,Y), $ $ \mathcal{B}_{\mu,0}^R(B_X,Y) $ and $ \mathcal{B}_{\mu,0}^{\rm aff}(B_X,Y)$ coincide. Moreover,
 \[ \|f\|_{\mathcal B_\mu^R(B_X,Y)} \le \|f\|_{\mathcal B_\mu^\nabla(B_X,Y)} \le 2\sqrt{2}R_\mu\|f\|_{\mathcal B_\mu^{\rm aff}(B_X,Y)}.  \]
\end{thm}
%


Next, we present a M\"obius invariant norm for the  Bloch-type space $W\mathcal{B}(B_X,Y). $

M\"obius transformations on a Hilbert space $ X $ are the mappings $ \varphi_a,  $ $ a \in B_X, $  defined as follows:
\begin{equation}\label{eq_Mobius}
 \varphi_a(z) = \frac{a - P_a(z) - s_aQ_a(z)}{1 - \langle z, a\rangle}, \quad z \in \B_X
\end{equation} 
where $s_a = \sqrt{1 - \|a\|^2},$   $P_a$ is the orthogonal projection from $X$ onto the one dimensional subspace $[a]$ generated by a, and $Q_a$ is the orthogonal projection from $X$ onto $X \ominus [a].$ It is clear that
$$ P_a(z) = \frac{\langle z, a\rangle}{\|a\|^2}a, \ (z \in X) \quad \text        {and} \quad Q_a(z) = z - \frac{\langle z, a\rangle}{\|a\|^2}a, \ (z \in B_X).  $$

When $a = 0,$ we simply define $\varphi _a(z)=-z.$ 
It is obvious that each $\varphi_a$  is a holomorphic mapping from $B_X$ into $X.$ 



We will also need the following facts about the pseudohyperbolic distance in $ B_X. $ It
is given by
\[ \varrho_X(x,y) := \|\varphi_{-y}(x)\| \quad \text{for any}\ x, y \in B_X. \]

For details concerning M\"obius transformations and the pseudohyperbolic distance we refer to the book of K. Zhu \cite{Zh}. 

%
%

It is well known that, in the case $ n\ge2, $ the equality $ \|f \circ \varphi\|_{\mathcal{B}^\nabla(\BB_n,Y)} =  \|f\|_{\mathcal{B}^\nabla(\BB_n,Y)} $ is   false.  Our goal is to  find a semi-norm on  $ \mathcal{B}(B_X,Y) $ which is invariant under the automorphisms
of the ball $ B_X. $

\begin{defn}
Let $ X $ be a complex Hilbert space, $ Y $ be a Banach space and $ f \in H(B_X,Y). $ Consider the invariant gradient norm
\[ \|\widetilde{\nabla}f(z)\| := \|\nabla(f \circ \varphi_z)(0)\|\quad\text{for any $z \in B_X. $}  \]
We recall the following result of Blasco and his colleagues in \cite{BGM}:
\begin{lem}[Lemma 3.5, \cite{BGM}]\label{lem_BGM}
	Let $  f\in H(B_X). $ Then
	\[ \|\widetilde{\nabla}f(z)\| = \sup_{w\neq 0}\dfrac{|\langle \nabla f(z), w\rangle| (1-\|z\|^2)}{\sqrt{(1-\|z\|^2)\|w\|^2 + |\langle w,z\rangle|^2}}. \]
\end{lem}

	We define invariant semi-norm   as follows
	\[ 		\|f\|_{s \mathcal B^{\rm inv}(B_X,Y)}  := \sup_{z\in B_X}\|\widetilde\nabla f(z)\| =\sup_{z\in B_X}\sup_{u\in Y', \|u\|\le1}\|\widetilde\nabla(u\circ f)(z)\|. \]
  We denote
	\[  		\mathcal B^{\rm inv}(B_X,Y) := \{f \in \mathcal B (B_X,Y):\ \|f\|_{s \mathcal B^{\rm inv}(B_X,Y)}  <\infty\}. \]
	
	It is also easy to check that $  \mathcal B^{\rm inv}(B_X,Y) $   is Banach under the  norm 
	\[ 		\|f\|_{\mathcal B^{\rm inv}(B_X,Y)}  := \|f(0)\| + \|f\|_{s\mathcal B^{\rm inv}(B_X,Y)}. \] 
\end{defn}

Now, applying Theorem 3.8 in  \cite{BGM} to the functions $ u\circ f $ for every $ u\in Y' $ we obtain the following:
\begin{thm}\label{thm_main1}
	The spaces $ \mathcal{B}^\nabla(B_X,Y), $ and $ \mathcal{B}^{\rm inv}(B_X,Y) 
	$  coincide. Moreover,
	\[ \|f\|_{\mathcal B^\nabla(B_X,Y)} \le \|f\|_{\mathcal B^{\rm inv}(B_X,Y)} \lesssim \|f\|_{\mathcal B^\nabla(B_X,Y)}.  \]
\end{thm}
Now let  $ W\subset Y' $ be a separating subspace of the dual $ Y'. $ 
Applying Proposition \ref{prop_4.3}, Theorems \ref{thm_main} and \ref{thm_main1} to  functions $w\circ f  $ for each $ f \in H(B_X,Y) $ and $ w\in W $ we obtain the equivalence of the norms in associated Bloch-type spaces:
\[ \|\cdot\|_{W\mathcal B_\mu^R(B_X)} \cong \|\cdot\|_{W\mathcal B_\mu^\nabla(B_X)} \cong \|\cdot\|_{W\mathcal B_\mu^{\rm aff}(B_X)},\]
\[ \|\cdot\|_{W\mathcal B^R(B_X)} \cong \|\cdot\|_{W\mathcal B^\nabla(B_X)} \cong \|\cdot\|_{W\mathcal B^{\rm aff}(B_X)} \cong \|\cdot\|_{W\mathcal B^{\rm inv}(B_X)}.\]


Hence, for the sake of simplicity, from now on we write $ \mathcal B_\mu $ instead of $ \mathcal B_\mu^R. $

\medspace
Now, we show that   $ W\mathcal{B}_\mu(B_X)(Y), $ $ W\mathcal{B}_{\mu,0}(B_X)(Y) $ 
 satisfy   (we1)-(we3).

We need the following lemma whose proof parallels that of Lemma 13 in \cite{SW} and will be omitted.

\begin{lem}\label{lem_weight} Let $ \mu $ be a   normal weight on $ B_X. $ Then there exists $C_\mu > 0 $ such that
	\[ C_\mu \le \frac{\mu(r)}{\mu(r^2)} \le 1\quad \forall r \in [0,1).\]
\end{lem}
%
\begin{prop}\label{prop_e1e3} Let
	  $W \subset Y'$ be a separating subspace. Let $ \mu $ be a   normal weight on $ B_X. $  Then $  \mathcal{B}_\mu(B_X), $ $  \mathcal{B}_{\mu,0}(B_X) $  satisfy (e1) and (e2), and hence,  $ W\mathcal{B}_\mu(B_X)(Y)$ and  $ W\mathcal{B}_{\mu,0}(B_X)(Y) $ satisfy  (we1)-(we3). \end{prop}
\begin{proof}
It is obvious that $  \mathcal{B}_\mu(B_X), $ $  \mathcal{B}_{\mu,0}(B_X) $ satisfy (e1).

Because  $ \mathcal B_{\mu,0}(B_X) $ is the subspace of $  \mathcal B_\mu(B_X) $
it suffices to check (e2) for the space $ \mathcal B_\mu(B_X). $

 In order to prove (e2) holds for $  \mathcal{B}_\mu(B_X) $ we will show that  the closed unit ball $ U $ of $ \mathcal{B}_\mu(B_X) $ is pointwise bounded and equicontinuous.

(i) First we prove that $ U $ is pointwise bounded.  It suffices to prove that
\begin{equation}\label{eq_5.1}
	|f(z)| \le \max\bigg\{1,\int_0^{\|z\|}\frac{dt}{\mu(t)}\bigg\}\|f\|_{\mathcal{B}_\mu(B_X)}\quad \forall f \in \mathcal{B}_\mu(B_X), \forall z \in B_X.
\end{equation} 

Fix $ f \in \mathcal{B}_\mu(B_X)  $ and put $ g(z) = f(z)-f(0) $ for every $ z \in B_X. $
  Note that $g(0) = 0$ and $\|g\|_{\mathcal{B}_\mu (B_X)} = \|f\|_{s\mathcal{B}_\mu (B_X)}.$ 
As in Lemma \ref{lem_Ti_4.11} by Cauchy-Schwarz inequality  we have
\[
|g(z)| \le  
  \int_0^1\frac{\|f\|_{s\mathcal{B}_\mu (\B_X)}\|z\|}{\mu(tz)}dt =  \|f\|_{s\mathcal{B}_\mu (B_X)}\int_0^{\|z\|}\frac{dt}{\mu(t)} = \|g\|_{\mathcal{B}_\mu (B_X)}\int_0^{\|z\|}\frac{dt}{\mu(t)}.  
\]
Consequently,
\[\begin{aligned}
|f(z)|  &\le |f(0)| + |g(z)| \le |f(0)| +\|g\|_{\mathcal{B}_\mu (B_X)}\int_0^{\|z\|}\frac{dt}{\mu(t)} \\
&= \|f\|_{\mathcal{B}_\mu(B_X)} - \|g\|_{\mathcal{B}_\mu(B_X)} + \|g\|_{\mathcal{B}_\mu(B_X)}\int_0^{\|z\|}\frac{dt}{\mu(t)}\\
&= \|f\|_{\mathcal{B}_\mu(B_X)} + \bigg(\int_0^{\|z\|}\frac{dt}{\mu(t)} - 1\bigg)\|f\|_{s\mathcal{B}_\mu(B_X)}\\
&\le \max\bigg\{1,\int_0^{\|z\|}\frac{dt}{\mu(t)}\bigg\}\|f\|_{\mathcal{B}_\mu(B_X)}.
\end{aligned}  \]

(ii) Next, we   show that $ U $ is equicontinuous. 
 For each $ f \in U, $ by Proposition  \ref{prop_OSm} we can find $ m \ge 2 $ such that 
\[ \|f\|_{s\mathcal{B}_\mu (B_X)} = \sup_{y\in OS_m}\|f_y\|_{s\mathcal{B}_\mu (\BB_m)}. \]
Fix $ e_{[m]} = (e_1, \ldots, e_m) \in OS_m. $ Then, for every $ z =(z_k)_{k\in \Gamma},w=(w_{k})_{k\in \Gamma} \in B_X, $ we consider $ z_{[m]} := (z_1, \ldots, z_m), $ $ w_{[m]} := (w_1, \ldots, w_m). $ By Theorem 3.6 in \cite{Zh} and Lemma \ref{lem_weight}  we have
\[\begin{aligned}
|f_{e_{[m]}}(z_{[m]}) &- f_{e_{[m]}}(w_{[m]})| \le \beta(z_{[m]}, w_{[m]})\sup_{x_{[m]}\in \BB_m} \|\widetilde{\nabla} f_{e_{[m]}}(x_{[m]})\| \\
&\le   \beta(z_{[m]}, w_{[m]})\sup_{x_{[m]}\in \BB_m}\sup_{y\in \BB_m\setminus \{0\}}\dfrac{|\langle \nabla f_{e_{[m]}}(x_{[m]}), y\rangle| (1-\|x_{[m]}\|^2)}{\sqrt{(1-\|x_{[m]}\|^2)\|y\|^2 + |\langle y,x_{[m]}\rangle|^2}} \\
&\le \beta(z_{[m]}, w_{[m]}) C_\mu^{-1}\sup_{x_{[m]}\in \BB_m}\dfrac{\mu^{[m]}(\|x_{[m]}\|)|\nabla f_{e_{[m]}}(x_{[m]})\sqrt{1-\|x_{[m]}\|^2}}{\mu^{[m]}(\|x_{[m]}\|^2)} \\
&\le \beta(z_{[m]}, w_{[m]}) C_\mu^{-1}\|f_{e_{[m]}}\|_{\mathcal B_\mu (\BB_m)}\dfrac{\sqrt{1-\|x_{[m]}\|^2}}{\mu^{[m]}(\|x_{[m]}\|^2)}
\end{aligned} \]
where $ \beta $ is the Bergman metric  on $ \BB_m $ given by
\[ \beta(s,t) = \frac12\log\dfrac{1+|(\varphi_m)_s(t)|}{1-|(\varphi_m)_s(t)|} \]
with $ (\varphi_m)_s $  is the involutive automorphism of $ \BB_m $ that interchanges $ 0 $ and $ s.$

If $ \|x_{[m]}\|^2\le \delta $ it is clear that 
\[ \dfrac{\sqrt{1-\|x_{[m]}\|^2}}{\mu^{[m]}(\|x_{[m]}\|^2)} \le \dfrac{1}{m_{\mu,\delta}} <\infty, \]
where $ m_{\mu,\delta} = \min_{t \in [0,\delta]}\mu(t) >0; $  if $ \|x_{[m]}\|^2 >\delta $ and $ b\ge 1/2 $ we have
\[ \dfrac{\sqrt{1-\|x_{[m]}\|^2}}{\mu^{[m]}(\|x_{[m]}\|^2)}\le \dfrac{(1-\|x_{[m]}\|^2)^b}{\mu^{[m]}(\|x_{[m]}\|^2)} < S_\mu <\infty; \]
if $ \|x_{[m]}\|^2 >\delta $ and $ b< 1/2 $ we get
 \[ \dfrac{\sqrt{1-\|x_{[m]}\|^2}}{\mu^{[m]}(\|x_{[m]}\|^2)} =  \dfrac{(1-\|x_{[m]}\|^2)^b}{\mu^{[m]}(\|x_{[m]}\|^2)}(1-\|x_{[m]}\|^2)^{1/2-b} \le S_\mu (1-\delta)^{1/2-b}<\infty.\]
Consequently, 
\[ |f_{e_{[m]}}(z_{[m]}) - f_{e_{[m]}}(w_{[m]})| \le  \beta(z_{[m]}, w_{[m]})\widehat{S}_\mu \|f_{e_{[m]}}\|_{\mathcal B_\mu(\BB_m)} \]
where
\[\widehat{S}_\mu := C_\mu^{-1}\max\{m_{\mu,\delta}^{-1}, S_\mu(1-\delta)^{1/2-b}\}. \]

Since $ \beta(s,t) $ is the infimum of the set consisting of all $ \ell(\gamma) $ where $ \gamma $ is a piecewise smooth curve in $ \BB_m $ from $ s $ to $ t $ (see \cite[p. 25]{Zh}) we have
\[ |f_{e_{[m]}}(z_{[m]}) - f_{e_{[m]}}(w_{[m]})| \le \|z_{[m]} - w_{[m]}\|\widehat{S}_\mu\|f_{e_{[m]}}\|_{\mathcal{B}_\mu(\BB_m)} \le \widehat{S}_\mu\|z  - w \|.  \]
Consequently,
\[ 	|f(z) - f(w)| = \lim_{m\to\infty}|f_{e_{[m]}}(z_{[m]}) - f_{e_{[m]}}(w_{[m]})|\le \widehat{S}_\mu\|z-w\|.
 \]
 This yields that $ U $ is equicontinuous.
\end{proof}

\begin{rmk}\label{rmk_1}
	(1) In fact, the estimate  (\ref{eq_5.1}) can be written as follows
	\[ |f(z)| \le |f(0)| + \int_0^{\|z\|}\frac{dt}{\mu(t)}\|f\|_{s\mathcal B_\mu}.\]
%
%
	
	(2) It should be noted that, for $ \mu(z) = 1-\|z\|^2, $ (\ref{eq_5.1}) will be
	\begin{equation}\label{eq_5.1'}
		|f(z)| \le \max\Bigg\{1, \dfrac{1}{2}\log\dfrac{1+\|z\|}{1-\|z\|}\Bigg\}\|f\|_{\mathcal{B}(B_X)}\quad \forall f \in \mathcal{B}(B_X), \forall z \in B_X, 
	\end{equation} 
	and, therefore, by an easy calculation that $\sup_{x \in [0,1)}(1-x^2)\log\frac{1+x}{1-x} <1 $ we have
	\begin{equation}\label{eq_5.2}
		(1 - \|z\|^2)|f(z)| \le \|f\|_{\mathcal B(B_X)}\quad \forall f \in \mathcal{B}(B_X), \forall z \in B_X.
	\end{equation}
\end{rmk}
%
%
%

\section{The Test Functions and Auxiliary results}
\setcounter{equation}{0}
Thí section provides some preparations to study characterizations the boundedness and compactness of the weighed composition operators between (little) Bloch-type spaces.

In this section we consider  $ \nu $ is a   normal weight on $ B_X $ and $\varphi   \in S(B_X).$

We begin this section by  constructing  test functions that are  usefull for the proofs of our main results.

%

First we consider the holomorphic function 
\begin{equation}\label{test_g} g(z) := 1 + \sum_{k > k_0}2^kz^{n_k} \quad \forall z \in \BB_1
\end{equation}   
where $ k_0 = \big[\log_2\frac{1}{\nu(\delta)}\big], $ $ n_k = \big[\frac{1}{1-r_k}\big] $ with $ r_k = \nu^{-1}(1/2^k) $ for every $ k\ge 1. $ Here the symbol $ [x] $ means the greatest integer not more than $ x. $ 
By \cite[Theorem 2.3]{HW}, $ g(t) $ is increasing on $ [0,1) $ and
\begin{equation}\label{test_g1}
	|g(z)| \le g(|z|) \in \R\quad \forall z \in \BB_1,
\end{equation}
\begin{equation}\label{test_g2}
	0< C_1 :=\inf_{t\in [0,1)}\nu(t)g(t) \le \sup_{t\in [0,1)}\nu(t)g(t)  \le  \sup_{z \in \BB_1}\nu(z)|g(z)| =: C_2 <\infty. 
\end{equation}
\begin{prop}\label{est_g1}
	There exists positive constants $ C_3 $ such that the inequality
	\begin{equation}\label{test_g2a}
		\int_0^r g(t)dt \le C_3 \int_0^{r^2}g(t)dt
	\end{equation}
	holds for all $ r \in [r_1, 1), $ where $ r_1 \in (0, 1) $ is a constant such that
	$  \int_0^{r_1}g(t)dt = 1. $
\end{prop}
\begin{proof}
	Let $ \delta >0 $ be the constant in Definition of the normal weight $ \nu. $ We may assume that $ r_1 <\delta^{1/4}. $ 
	
	In the case $ r \in [r_1, \delta^{1/4}] $ we have $ \int_0^rg(t)dt $ is bounded above
	and $ \int_0^{r^2}g(t)dt $ is bounded below by a positive constant. So, there exists a constant $ C > 0 $ such that
	\[ \int_0^rg(t)dt  \le C  \int_0^{r^2}g(t)dt, \quad r \in [r_1, \delta^{1/4}].\]
	
	In the case $ r \in (\delta^{1/4}, 1), $ by (\ref{test_g2}) and (\ref{w_2}) we have
	\begin{equation}\label{ineq_g1}
		\begin{aligned}
			\int_{r^2}^rg(t)dt &\le C_2\int_{r^2}^r\dfrac{1}{\nu(t)}dt = C_2\int_{r^2}^r\dfrac{(1-t)^b}{\nu(t)}\dfrac{1}{(1-t)^b}dt \\
			&\le C_2\frac{(1-r^2)^b}{\nu(r^2)}\frac{r-r^2}{(1-r)^b} \\
			&= C_2\frac{(1-r^2)^b}{\nu(r^2)}\frac{(r-r^2)(1+r)^b(1+r^2)^b}{(1-r^4)^b} \\
			&\le  C_2\frac{(r-r^2)(1+r)^b(1+r^2)^b}{r^2-r^4}\int_{r^4}^{r^2}\frac{(1-t)^b}{\nu(t)}\frac{1}{(1-t)^b}dt \\
			&\le \frac{C_2(1+r)^b(1+r^2)^b}{C_1(r+r^2)}\int_{r^4}^{r^2}g(t)dt.
	\end{aligned} \end{equation}
	Therefore, there exists a constant $ C_3> 0 $ such that
	\[ \int_0^rg(t)dt = \int_0^{r^2}g(t)dt + \int_{r^2}^rg(t)dt \le C_3\int_0^{r^2}g(t)dt,\quad r \in (\delta^{1/4}, 1). \]
	This implies that (\ref{test_g2a}) is proved.
\end{proof}

Now, for $ w \in B_X $ fixed such that $ \|w\| \ge r $ for some $ r>0,$ consider  put the test functions
\begin{equation}\label{test_func2}
	\beta_{w}(z) := \dfrac{1}{\|w\|}\int_0^{\langle z,w\rangle}g(t)dt, \quad z\in B_X, 
\end{equation} 
\begin{equation}\label{test_func3}
	\gamma_{w}(z) := \dfrac{1}{\int_0^{\|w\|^2}g(t)dt}\Bigg(\int_0^{\langle z,w\rangle}g(t)dt\Bigg)^2,\quad z\in B_X,
\end{equation} 

Let $ \{w^n\}_{n\ge1} \subset B_X $ be a sequence satisfying, for some $ r>0, $ $ \|w^n\| \ge r $  for every $ n\ge 1 $ and $ \lim_{n\to \infty}w^n = w^0 $ with $ \|w^0\| = 1. $ Under the additional assumption   that 
$$ \int_0^1\dfrac{1}{\nu(t)}< \infty, $$
for each   $n\ge1 $ consider the sequences the functions  $ \{\gamma_{w^n}\}_{n\ge1} $ defined by  (\ref{test_func3}) and of the functions
\begin{equation}\label{test_func4}\begin{aligned}
		\eta_{w^n}(z) &:= \frac{1}{\|w^n\|^{n+1}}\int_0^{\|w^n\|^n\langle z, w^n\rangle}g(t)dt, \quad n =  1,2,\ldots, \\
		\theta_{w^n} &:=  \eta_{w^n} - \beta_{w^n}\quad n =  1,2,\ldots.
\end{aligned}\end{equation}
\begin{prop}\label{lem2} 
	We have 	 $ \beta_{w}, \gamma_{w}, \theta_{w} \in \mathcal B_{\nu,0}(B_X) $ and 
	$$ \|\beta_{w}\|_{\mathcal B_\nu(B_X)} \le C_2, \, \|\gamma_{w}\|_{\mathcal B_\nu(B_X)} \le 2C_2C_3, \, \|\theta_{w^n}\|_{\mathcal B_\nu(B_X)} \le 2C_2. $$
\end{prop}
\begin{proof}
	It suffices to prove for the function $ \beta_{w} $ because the proofs for other ones are very similar.
	
	(i) Since $ \nu(z) \to 0, $ from (\ref{test_g2}) we have 
	\[\begin{aligned}
		\nu(z)\|\nabla \beta_{w}(z)\| &= \dfrac{1}{\|w\|}|\langle z, w\rangle|\nu(z)| |g(\langle z, w\rangle)| \\
		&\le \nu(z) g(\|w\|) \to 0\quad \text{as}\ \|z\| \to 1.
	\end{aligned}
	\]
	This implies   that $ \beta_{w} \in \mathcal B_{\nu,0}(B_X). $ 
	
	(ii) Obviously, $ \beta_{w}(0) = 0. $ On the other hand, we have
	\[\begin{aligned}
		\|\beta_{w}\|_{\mathcal B_\nu(B_X)} & = \sup_{z\in B_X}\nu(z)\|R\beta_{w}(z)\| \le \nu(z)\dfrac{1}{\|w\|}|\langle z,w\rangle||g(\langle z, w\rangle)| \le C_2 <\infty.
	\end{aligned} \]
\end{proof}
\begin{prop}\label{lem3} We have
	\begin{enumerate}[\rm (1)]
		\item  The sequences $ \{\theta_{w^n}\}_{n\ge 1}, $  $ \{\gamma_{w^n}\}_{n\ge1} ,$ $ \{\widetilde{\eta}_{w^n}\}_{n\ge1} $  are bounded in $ \mathcal B_\nu(B_X); $ 
		\item $ 	\gamma_{w^n} \to 0 $  
		uniformly
		on any compact subset of $ B_X$  if $ \int_0^1\frac{dt}{\nu(t)} =\infty; $
		\item  $ \theta_{w^n} \to0 $
		uniformly  	on any compact subset of $ B_X$ if  $ \int_0^1\frac{dt}{\nu(t)} <\infty. $
	\end{enumerate}
\end{prop}
\begin{proof} (1) It follows from Proposition \ref{lem2}. 
	
	(2) By (\ref{test_g2}) and  the assumption $ \int_0^1\frac{1}{\nu(t)} = \infty, $ we  have $ \int_0^1g(t)dt = \infty. $ Then it is easy to check that $ 	\gamma_{w^n} \to 0 $  
	uniformly
	on any compact subset of $ B_X.$   
	
	(3) First, note that, by the assumption $ \int_0^1\frac{1}{\nu(t)} <\infty $ we have $ \int_0^1g(t )dt < \infty.$ 
	Let $ K $ be a compact subset of $ B_X. $ We may assume that $ \sup_{z \in K}\|z\| \le R <1. $ For every $ z\in K $ we have
	\[ \begin{aligned}
		|\theta_{w^n}(z)| &\le   \dfrac{1}{\|w^n\|^{n+2}}\bigg[\|w^n\|^{n+1}\int_0^{\langle z, w^n\rangle}g(t)dt-\|w^n\|\int_0^{\|w^n\|^n\langle z, w^n\rangle}g(t)dt\bigg]  \\
	&\le   \dfrac{1}{\|w^n\|}\bigg[\int_0^{\langle z, w^n\rangle}g(t)dt- \int_0^{\|w^n\|^n\langle z, w^n\rangle}g(t)dt\bigg]  \\
		&= 
		\dfrac{1}{\|w^n\|}\int_{\|w^n\|^n\langle z, w^n\rangle}^{\langle z, w^n\rangle}g(t)dt\\
		&\le  \dfrac{1}{\|w^n\|}\|z\| \|1-w^n\|\sup_{t\in [r,R]}|g(t)| \\
		&\le \dfrac{1}{\|w^n\|}R\|1-w^n\|\sup_{t\in [r,R]}|g(t)| \\
		&\to 0 \quad\text{as}\ n \to \infty.
		\end{aligned} \]
	And we are done.
\end{proof}

We also  need the following lemma which is easily obtained   by using  Mincowski's inequality and elementary calculations:
\begin{lem}\label{nabla} Let $\psi, f \in H(B_X) $ and  $\varphi \in S(B_X).$ Then , for every $ z \in B_X $ we have
	\begin{equation}\label{nabla_1} \|\nabla_z(\psi.(f \circ \varphi))(z)\|  \le |f(\varphi(z))|\|\nabla_z\psi(z)\| +  |\psi(z)|\big\|\nabla_{\varphi(z)}f(\varphi(z))\big\|\|\nabla_z\varphi(z)\| 
	\end{equation}
	and
	\begin{equation}\label{R_1}  \|R(\psi\cdot(f \circ \varphi))(z)\|  \le |f(\varphi(z))|\|R\psi(z)\| +  |\psi(z)|\big\|\nabla_{\varphi(z)}f(\varphi(z))\big\|\|R\varphi(z)\|.
	\end{equation}
\end{lem}

In order to study the compactness of  the operators  $W_{\psi. \varphi},$   as in \cite{Tj} we can prove the following.
\begin{lem}\label{lem_Tj}
	Let $E, F$ be two Banach spaces of holomorphic functions on $B_X.$ Suppose 
	that
	\begin{enumerate}[\rm(1)]
		\item The point evaluation functionals on $E$ are continuous;
		\item The closed unit ball of $E$ is a compact subset of $E$ in the topology of uniform
		convergence on compact sets;
		\item  $T : E \to F$ is continuous when $E$ and $F$ are given the topology of uniform
		convergence on compact sets.
	\end{enumerate}
	Then, $T$ is a compact operator if and only if given a bounded sequence $\{f_n\}$ in $ E $ such
	that $f_n \to 0$ uniformly on compact sets, then the sequence $\{Tf_n\}$ converges to zero in the norm of $F.$
\end{lem}

We can now combine this result with  Montel theorem and  (\ref{eq_5.1}) to  obtain the following proposition. The details of the proof are omitted here.
\begin{prop}\label{prop_W_compact} Let $\psi \in H(B_X)$ and $\varphi\in S(B_X).$   Then  $W_{\psi, \varphi}: \mathcal B_{\nu}(B_X) \to \mathcal B_{\mu}(B_X)$ is compact if and only if $\|W_{\psi, \varphi}(f_n)\|_{\mathcal B_\mu(B_X)} \to 0$ for any bounded sequence $\{f_n\}$ in $\mathcal B_{\nu}(B_X)$ converging to $0$ uniformly on  compact sets in $B_X.$ 
\end{prop}

\section{Weighted composition operators between Bloch-type spaces}
\setcounter{equation}{0}

Let $ \nu, \mu $ be  normal weights on $ B_X. $
Let $\varphi \in S(B_X),$  the set of  holomorphic self-maps on $ B_X$ and $ \psi \in H(B_X). $ For each   $ F \subset \Gamma $  finite and $ m \in \N $ we write
\[ \varphi^{[F]} = \varphi\big|_{\text{span}\{e_k, k \in F\}}, \quad \varphi^{[m]} = \varphi\big|_{\text{span}\{e_1, \ldots, e_m\}}\]
and
\[ \psi^{[F]} = \psi\big|_{\text{span}\{e_k, k \in F\}}, \quad \psi^{[m]} = \psi\big|_{\text{span}\{e_1, \ldots, e_m\}}.\]
For each $ j \in \Gamma $ and $ k\ge 1 $ we denote
\[ \varphi_j(\cdot) := \langle\varphi(\cdot), e_j\rangle, \quad \varphi_{(k)}(\cdot) := (\varphi_1(\cdot), \ldots, \varphi_k(\cdot)).  \]
      In this section we investigate the boundedness and the compactness of  the   operators $ W_{\psi,\varphi} $ between the (little) Bloch-type spaces $ \mathcal B_\nu $ ($ \mathcal B_{\nu,0} $)  and $ \mathcal B_\mu,$    ($ \mathcal B_{\mu,0}$)  via the estimates of $ \psi^{[m]}, \varphi^{[m]} $ and $ \varphi_{(k)}. $ 
Hence, by Theorem \ref{thm_WO}, some characterizations for the boundedness and compactness of the operators $ \widetilde{W}_{\psi, \varphi} $ between  spaces $ W\mathcal B_\nu(B_X,Y), $ ($ W\mathcal B_{\nu,0}(B_X,Y) $) and  $ W\mathcal B_{\mu}(B_X,Y), $ ($ W\mathcal B_{\nu,0}(B_X,Y) $)  will be obtained from these results. 

By Proposition \ref{prop_4.3} and Theorem \ref{thm_main} in the paper we will only present the results for the spaces $ \mathcal B_\mu^R(B_X). $

First we investigate the boundedness of weighted composition operators.
 We use there certain quantities, which will be used in this work. We list them below:
 
 \begin{subequations}
 	\begin{align}
 		 		&\mu^{[m]}(y)\|R\psi^{[m]}(y)\|\max\bigg\{1, \displaystyle\int_0^{\|\varphi^{[m]}(y)\|}\dfrac{dt}{\nu(t)}\bigg\}, \label{B_1a}\\
  &\dfrac{\mu^{[m]}(y)|\psi^{[m]}(y)|\|R\varphi^{[m]}(y)\|}{\nu(\varphi^{[m]}(y))}, \label{B_1b}
\end{align}
\end{subequations}
 \begin{subequations}
	\begin{align}
	&\mu^{[F]}(y)\|R\psi^{[F]}(y)\| \max\bigg\{1, \int_0^{\|\varphi^{[F]}(y)\|}\dfrac{dt}{\nu(t)}\bigg\}, \label{B_2a}\\
&\dfrac{\mu^{[F]}(y)|\psi^{[F]}(y)|\|R\varphi^{[F]}(y)\|}{\nu(\varphi^{[F]}(y))}, \label{B_2b}
	\end{align}
\end{subequations}
 \begin{subequations}
	\begin{align}
	&\mu(y)\|R\psi(y)\| \max\bigg\{1, \displaystyle\int_0^{\|\varphi_{(k)}(y)\|}\dfrac{dt}{\nu(t)}\bigg\}, \label{B_3a}\\
&\frac{\mu(y)|\psi(y)|\|R\varphi_{(k)}(y)\|}{\nu(\varphi_{(k)}(y))}, \label{B_3b}
	\end{align}
\end{subequations}
  \begin{subequations}
 	\begin{align}
 		&\mu(y)\|R\psi(y)\| \max\bigg\{1, \displaystyle\int_0^{\|\varphi(y)\|}\dfrac{dt}{\nu(t)}\bigg\}, \label{B_4a}\\
 	&\frac{\mu(y)|\psi(y)|\|R\varphi(y)\|}{\nu(\varphi(y))}. \label{B_4b}
 	\end{align}
 \end{subequations}
 
\begin{thm}
	\label{thm6_1} Let $\psi \in H(B_X)$ and $\varphi   \in S(B_X).$   Let  $ \mu, \nu $ be   normal weights on $ B_X. $  Then   the  following are equivalent:  
	\begin{enumerate}[\rm(1)]
		\item 	$M^{[m]}_{R\psi,\varphi} := \displaystyle\sup_{y \in \BB_m}(\ref{B_1a})<\infty,$ 
		$ M^{[m]}_{\psi,R\varphi} := \displaystyle\sup_{y \in \BB_m} (\ref{B_1b})<\infty $ for some $ m \ge 2;$ 
		\item $ M^{[F]}_{R\psi,\varphi} := \displaystyle\sup_{y \in \BB_{[F]}}(\ref{B_2a})<\infty, $ $ M^{[F]}_{\psi,R\varphi} := \displaystyle\sup_{y \in \BB_{[F]}}(\ref{B_2b})<\infty $ for every $ F \subset \Gamma$ finite;
			\item $  M^{(k)}_{R\psi,\varphi} := \displaystyle\sup_{y \in B_X} (\ref{B_3a})<\infty, $ $ M^{(k)}_{\psi,R\varphi} := \displaystyle\sup_{y \in B_X}(\ref{B_3b}) <\infty $ for every $ k\ge 1;$ 
			\item $  M_{R\psi,\varphi} := \displaystyle\sup_{y \in B_X}(\ref{B_4a}) <\infty,$  $ M_{\psi,R\varphi}:= \displaystyle\sup_{y \in B_X}(\ref{B_4b}) <\infty;$ 
		\item $W_{\psi, \varphi}:   \mathcal B_\nu(B_X) \to \mathcal B_\mu(B_X)$     is bounded; 
		\item $W^0_{\psi, \varphi}: \mathcal B_{\nu,0}(B_X) \to \mathcal B_\mu(B_X)$     is bounded. 
	\end{enumerate}
	Moreover, if  $W_{\psi, \varphi}$     is bounded,  the following asymptotic relation holds for some $ m\ge 2: $
		\begin{equation} \label{eq_Thm6.2}
			\|W^{0,0}_{\psi,\varphi}\|    \asymp |\psi(0)|\max\bigg\{1,\int_0^{\|\varphi(0)\|}\frac{dt}{\mu(t )}\bigg\} + M_{R\psi,\varphi} + 
			M_{\psi, R\varphi},
		\end{equation}
	 \end{thm}
\begin{proof}
It is clear  that (4) $ \Rightarrow $ (2) $ \Rightarrow $ (1) and (4) $ \Rightarrow $ (3). Since $ \mathcal B_{\nu,0}(B_X) \subset \mathcal B_\nu(B_X)$    the implication (5) $\Rightarrow$ (6)   is obvious.

\medskip
(1) $\Rightarrow$ (5): Assume that $M^{[m]}_{R\psi, \varphi} < \infty$ and  $M^{[m]}_{\psi,R\varphi} < \infty $ for some $ m\ge 2. $ For each $ x \in OS_m $ we write $ z_x := \sum_{k=1}^mz_kx_k. $ Note that $ \|z_x\| = \|z_{[m]}\| $ and hence $ \mu^{[m]}(z_{[m]}) = \mu^{[m]}(z_x). $
By (\ref{eq_5.1}) 
and
 (\ref{nabla_1}), for every $ f \in \mathcal{B}_\nu(B_X) $ 
 and for every $ x \in OS_m $
  we have 
$$
\aligned
&\|(W_{\psi,\varphi}(f))_x\|_{s\mathcal{B}_{\mu^{[m]}}(\BB_m)} =    \sup_{z_x \in \BB_m}\mu^{[m]}(z_x)\|R\psi.(f \circ\varphi)]_x(z_{[m]})\| \\
&= \sup_{z_x \in \BB_m}\mu^{[m]}(z_x)\Big[|(f\circ\varphi)_x(z_{[m]}))|\|R\psi^{[m]}(z_x)\| \\
&\quad +  |\psi^{[m]}(z_x)|\big\|\nabla_{\varphi^{[m]}(z_x)}(f (\varphi^{[m]}(z_x))\big\|\|R\varphi^{[m]}(z_x)\|\Big]  \\
&= \sup_{z_x \in \BB_m}\mu^{[m]}(z_x)\Big[|f(\varphi^{[m]}(z_x))|\|R\psi^{[m]}(z_x)\| \\
&\quad +  |\psi^{[m]}(z_x)|\|\nabla_{\varphi^{[m]}(z_x)}f(\varphi^{[m]}(z_x))\|\|R\varphi^{[m]}(z_x)\|\Big]  \\
&\le \sup_{z_x \in \BB_m} \mu^{[m]}(z_x)\|R\psi^{[m]}(z_x)\|  \max\bigg\{1, \int_0^{\|\varphi^{[m]}(z_x)\|}\dfrac{dt}{\nu(t)}\bigg\}\|f\|_{\mathcal B_\nu(B_X)} \\
&\quad + \sup_{z_x \in \BB_m}\mu^{[m]}(z_x) |\psi^{[m]}(z_x)||\nabla_{\varphi^{[m]}(z_x)}f(\varphi^{[m]}(z_x))|\|R\varphi^{[m]}(z_x)\| \\
&\le \sup_{z_x \in \BB_m} \mu^{[m]}(z_x)\|R\psi^{[m]}(z_x)\|  \max\bigg\{1, \int_0^{\|\varphi^{[m]}(z_x)\|}\dfrac{dt}{\nu(t)}\bigg\}\|f\|_{\mathcal B_\nu(B_X)} \\
&\quad + \sup_{z_x \in \BB_m}\frac{\mu^{[m]}(z_x) |\psi^{[m]}(z_x)|\|R\varphi^{[m]}(z_x)\|}{\nu^{[m]}(\varphi^{[m]}(z_x))} \nu^{[m]}(\varphi^{[m]}(z_x)) \|\nabla_{\varphi^{[m]}(z_x)}f(\varphi^{[m]}(z_x))\| \\
 &\lesssim (M^{[m]}_{R\psi, \varphi} + M^{[m]}_{\psi, R\varphi})\|f\|_{\mathcal B_\nu(B_X)}.
\endaligned$$
Consequently, by (\ref{eq_Prop4.1}) we have
	\begin{equation}\label{eq_Thm6.1a}
\aligned
\|W_{\psi, \varphi}(f)\|_{s\mathcal B_\mu(B_X)} &=\sup_{x \in OS_m}\|(W_{\psi,\varphi}(f))_x\|_{s\mathcal{B}_\mu^{[m]}(\BB_m)} \\
&\lesssim (M^{[m]}_{R\psi, \varphi} + M^{[m]}_{\psi, R\varphi})\|f\|_{\mathcal B_\nu(B_X)}.
\endaligned
\end{equation}
Thus, (5) is proved.  

\medskip
(3) $\Rightarrow$ (5): Use  an argument analogous and an estimate to the previous one.

\medskip
(6) $\Rightarrow$ (4): Suppose $W^0_{\psi, \varphi}: \mathcal B_{\nu,0}(B_X) \to \mathcal B_{\mu}(B_X)$ is bounded.

First, it is obvious that $\psi \in \mathcal B_\mu(B_X).$ Indeed, by considering the function $f = 1 \in \mathcal B_{\nu,0}(B_X)$ it easy to see that
$$ \sup_{z \in B_X}\mu(z)|R\psi(z)| = \|\psi\|_{s\mathcal B_\mu(B_X)}  = \|W^1_{\psi,\varphi}(1)\|_{\mathcal B_\mu(B_X)} \le \|W^0_{\psi, \varphi}\|  < \infty. $$

Now, 
we will prove $ M_{\psi, R\varphi} < \infty.$ 

For each $ k \in \Gamma $
consider the test function $ \gamma_k \in   \mathcal B_{\nu}(B_X)  $ given by $ \gamma_{k}(z) := z_k $ for every $ z \in B_X. $ Then, since $ W^0_{\psi,\varphi} $ is bounded we have
\[\begin{aligned}
	\|W^0_{\psi,\varphi}(\gamma_k)\|_{\mathcal B_\mu(B_X)} &= \sup_{z\in B_X}\mu(z)\|R(\psi \cdot(\gamma_k\circ \varphi))(z)\| \\
	&= \sup_{z\in B_X}\mu(z)\|R(\psi\cdot\varphi)(z)\| \\
	&= \sup_{z\in B_X}\mu(z)\|R\psi(z)\varphi(z) + \psi(z)R\varphi(z)\|.
\end{aligned}  \]
This implies that $ \psi\cdot\varphi_k \in \mathcal B_{\mu}(\B_X). $

Assume the contrary, that $ M_{\psi,R\varphi} = \infty. $   Then there exists   $ \{z^n\}_{n\ge1} \subset B_X $ such that
\[ \frac{\mu(z^n)|\psi(z^n)|
	\|R\varphi(z^n)\|}{\nu(\varphi(z^n))} \ge n \quad \forall n\ge1.\]
We claim that $ \liminf_{n\to\infty}\|\varphi(z^n)\| := r > 0. $ Indeed, if this was not the case, since $\nu$ is positive, continuous and $ \lim_{\|z\|\to 1}\nu(z) = 0 $ we get $ \limsup_{n\to \infty}\nu(\varphi(z^n)) > 0. $ Therefore, from the estimate
\[\begin{aligned}
	\|\mu(z^n)R\psi&(z^n)\varphi(z^n) + \mu(z^n)\psi(z^n)R\varphi(z^n)\| \\
	&\ge \|\mu(z^n)\psi(z^n)R\varphi(z^n)\| - \|\mu(z^n)R\psi(z^n)\varphi(z^n)\|  \to \infty,
\end{aligned} \]
we would have $ \psi\cdot\varphi_k \notin \mathcal B_{\mu}(B_X) $ for some $ k \in \Gamma. $ We obtain a contradiction.

Therefore, we may assume that there exists some $ \eta > 0  $ such that $ |\varphi(z^n)| >\eta $ for every $ n\ge1. $ Then,
putting $ w^n = \varphi(z^n)$ we can consider the test function $\beta_{w^n}$ given by  (\ref{test_func2}).
Then, since $ \beta_{w^n}(0) = 0, $ we have 
\[ \nu(z)\|R\beta_{w^n}(z)\| = \nu(z)\|w^n\| |g(\langle z, w^n\rangle)| \le \nu(z)g(\|z\|) \le C_2. \]
By the boundedness of $ W^0_{\psi,\varphi}, $ there exists $ C_W>0 $ such that
\begin{equation}\label{eq_bounded}
	\|W^0_{\psi,\varphi}(\beta_{w^n})\|_{\mathcal B_\mu(B_X)} \le C_W\|\beta_{w^n}\|_{\mathcal B_\nu(B_X)} \le C_WC_2.
\end{equation} 
By Lemma \ref{lem_weight} we have
\begin{equation}\label{eq_frac_nu}
	\frac{\nu(\|w^n\|)}{\nu(\|w^n\|^2)} \ge C_\mu\quad \forall n\ge 1.
\end{equation} 
This implies that
\begin{equation} \begin{aligned}\label{est_contrad_1}
		\|W^0_{\psi,\varphi}&(\beta_{w^n})\|_{\mathcal B_\mu(B_X)} \\
		&\ge \|\mu(z^n)R\psi(z^n)\beta_{w^n}(\varphi(z^n)) \\
		&\quad+ \mu(z^n)\psi(z^n)\nabla\beta_{w^n}(\varphi(z^n))R\varphi(z^n)\| \\
		&\ge \|\mu(z^n)\psi^{[F]}(z^n)\nabla\beta_{w^n}(\varphi(z^n))R\varphi(z^n)\| \\
		&\quad - \|\mu(z^n)R\psi(z^n)\beta_{w^n}(\varphi(z^n))\| \\
		&\ge \dfrac{\mu(z^n)|\psi(z^n)| \|R\varphi(z^n)\|}{\nu(\|\varphi(z^n)\|)}\nu(\|w^n\|^2)g(\|w^n\|^2)\dfrac{\nu(\|w^n\|)}{\nu(\|w^n\|^2)} \\
		&\quad - \|\mu(z^n)R\psi(z^n)\beta_{w^n}(\varphi(z^n))\| \\
		&\ge n C_1C_\mu -  \|\mu(z^n)R\psi(z^n)\beta_{w^n}(\varphi(z^n))\|. 
	\end{aligned} 
\end{equation}
On the other hand, 
\[\begin{aligned}
\|\mu(z^n)&R\psi(z^n)\beta_{w^n}(\varphi(z^n))\| \\
	&\le 	\|W^0_{\psi,\varphi}(\beta_{w^n})\|_{\mathcal B_\mu(B_X)}  + \|\psi\|_{\mathcal B_{\mu}(B_X)}\nu(\|w^n\|^2)|g(\|w^n\|^2) \\
	&\le \|W^0_{\psi,\varphi}(\beta_{w^n})\|_{\mathcal B_\mu(B_X)}  + \|\psi\|_{\mathcal B_{\mu}(B_X)}C_2 \\
	&\le C_2C_W + \|\psi\|_{\mathcal B_{\mu}(B_X)}C_2.
\end{aligned}\]
Thus, 
\[ 2C_2C_W  + \|\psi\|_{\mathcal B_{\mu}(B_X)}C_2  \ge nC_1C_\mu \to \infty\quad \text{as}\ n \to \infty.\]
This is a contradiction to (\ref{eq_bounded}).

Finally, we prove $ M_{R\psi,\varphi} < \infty.$
Otherwise, we can find  $ \{z^n\}_{n\ge1} \subset B_X $ such that 
\[ \mu(z^n)\|R\psi(z^n)\|\int_0^{\|\varphi(z^n)\|}\dfrac{dt}{\nu(t)} > n. \]
By $ \psi \in \mathcal B_{\mu}(B_X) $ this implies that 
\[ \int_0^{\|\varphi(z^n)\|}\dfrac{dt}{\nu(t)} > \frac{n}{\|\psi\|_{\mathcal B_{\mu}(B_X)}}. \]
Put $ w^n = \varphi(z^n). $ 
Then, by (\ref{test_g2}), for every $ n\ge1 $ we have
\begin{equation}\label{est_gamma2}
	\begin{aligned}
		\dfrac{1}{C_1}\int_0^{\|w^n\|}g(t)dt   \ge \int_0^{\|w^n\|}\dfrac{dt}{\nu(t)} > \frac{n}{\|\psi\|_{\mathcal B_{\mu}(B_X)}}.
	\end{aligned} 
\end{equation}  
Therefore, $ \sup_{n\ge1}\int_0^{\|w^n\|}g(t)dt = \infty,$
hence,
\[ \liminf_{n\to\infty}\|\varphi(z^n)\| = \liminf_{n\to\infty}\|w^n\|  > 0. \]
Then, we can consider the   sequence $ \{\gamma_{w^n}\}_{n\ge1} $ defined  by (\ref{test_func3}). 	By (\ref{est_gamma2}) and (\ref{test_g2a}), for every $ n\ge 1, $ we have
\begin{equation}\label{es_gamma_w1}
\begin{aligned}
	|\gamma_{w^n}(\varphi(z^n))| &=  \Bigg|\dfrac{\int_0^{\langle w^n,w^n\rangle}g(t)dt}{\int_0^{\|w^n\|^2}g(t)dt} \int_0^{\langle w^n,w^n\rangle}g(t)dt\Bigg| \\
	&\ge \Bigg|\dfrac{\int_0^{\langle w^n,w^n\rangle}g(t)dt}{\int_0^{\|w^n\|^2}g(t)dt}\frac{1}{C_3} \int_0^{\|w^n\|}g(t)dt\Bigg| \\
	&\ge \Bigg|\dfrac{\int_0^{\|w^n\|}g(t)dt}{\int_0^{\|w^n\|^2}g(t)dt}\Bigg| \frac{C_1}{C_3}\int_0^{\|w^n\|}\dfrac{dt}{\nu(t)} \\
	&=  \frac{C_1}{C_3}\int_0^{\|w^n\|}\dfrac{dt}{\nu(t)}.
\end{aligned} 
\end{equation}
Then, since $W^0_{\psi,\varphi}$  is bounded and $M_{\psi,R\varphi} < \infty,$ for every $ n\ge 1 $ we have
	\begin{equation}\label{es_gamma_w2}
	 \begin{aligned}
	n<	\mu(z^n)&\|R\psi(z^n)\|\int_0^{\|w^n\|}\dfrac{dt}{\nu(t)} \\
		&\le \frac{C_3}{C_1}\mu(z^n)\|R\psi(z^n)\||\gamma_{w^n}(\varphi(z^n))| \\
		&\le  \frac{C_3}{C_1}\bigg[\mu(z^n)\|RW^0_{\psi, \varphi}\gamma_{w^n}(z^n)\|  + \frac{\mu(z^n)|\psi(z^n)|\|R\varphi(z^n)\|}{\nu(\varphi(z^n))}\|\gamma_{w^n}\|_{\mathcal B_\nu(B_X)}\bigg] \\
		 &\le  \frac{C_3}{C_1}\bigg[\|W^0_{\psi, \varphi}\gamma_{w^n}\|_{\mathcal B_\mu(X)}  + \frac{\mu(z^n)|\psi(z^n)|\|R\varphi(z^n)\|}{\nu(\varphi(z^n))}\|\gamma_{w^n}\|_{\mathcal B_\nu(B_X)}\bigg] \\
		 	 &\le  \frac{C_3}{C_1}\bigg[\|W^0_{\psi, \varphi}\| \|\gamma_{w^n}\|_{\mathcal B_\nu(X)}  + \frac{\mu(z^n)|\psi(z^n)|\|R\varphi(z^n)\|}{\nu(\varphi(z^n))}\|\gamma_{w^n}\|_{\mathcal B_\nu(B_X)}\bigg] <\infty.
	\end{aligned} 
	\end{equation} 
It is impossible. Therefore, $ M_{R\psi,\varphi}<\infty. $

%

Finally,  combining    the above  estimates   gives  (\ref{eq_Thm6.2}). 
The proof of Theorem is completed.
\end{proof}

Next, we will touch the characterizations for  the boundedness of the operators from $ \mathcal B_{\nu,0}(B_X)$  to  $\mathcal B_{\mu,0}(B_X).$

\begin{thm}\label{thm2_2} Let $\psi \in H(B_X)$ and $\varphi   \in S(B_X).$ 
	   Let  $ \mu, \nu $ be   normal weights on $ B_X. $  Then   the  following are equivalent:
	\begin{enumerate}[\rm(1)]
		\item  $\psi^{[m]}, \psi^{[m]} \cdot\varphi^{[m]}_j \in \mathcal B_{\mu^{[m]},0}(\BB_m),$ 
		$M^{[m]}_{R\psi,\varphi}<\infty, $ $M^{[m]}_{\psi,R\varphi}  <\infty$  for some  $ m \ge 2$ and  for every $ j\in \Gamma; $ 
			\item $\psi^{[F]}, \psi^{[F]} \cdot\varphi^{[F]}_j \in \mathcal B_{\mu^{[F]},0}(\BB_{[F]}),$  
			$ M^{[F]}_{R\psi,\varphi} <\infty,$  $ M^{[F]}_{\psi,R\varphi} <\infty$
			for every  $ F\subset \Gamma$ finite and  for every $ j\in \Gamma; $
					\item    $\psi, \psi \cdot\varphi_j \in \mathcal B_{\mu,0}(B_X),$  $  M^{(k)}_{R\psi,\varphi} <\infty, $ $ M^{(k)}_{\psi,R\varphi} <\infty $    for every $ k\ge1, $ and for every $ j\in \Gamma; $  
						\item    $\psi, \psi \cdot\varphi_j \in \mathcal B_{\mu,0}(B_X),$ $  M_{R\psi,\varphi}  <\infty,$  $ M_{\psi,R\varphi}  <\infty;$    for every $ j\in \Gamma; $  
		\item  $W^{0,0}_{\psi, \varphi}: \mathcal B_{\nu,0}(B_X) \to \mathcal B_{\mu,0}(B_X)$     is bounded. 	
	\end{enumerate}
	In this case, the   asymptotic relation  for 	$ \|W^{0,0}_{\psi,\varphi}\|  $ is as (\ref{eq_Thm6.2}).  
\end{thm}

\begin{proof}
It is obvious  that (4) $ \Rightarrow $ (2) $ \Rightarrow $ (1) and (4) $ \Rightarrow $ (3).

\medspace
	(1) $ \Rightarrow $  (5): 
	First, we note that   if  $ \psi^{[m]}, \psi^{[m]}\cdot\varphi^{[m]}_j \in \mathcal B_{\mu^{[m]},0}(\BB_m) $   it follows from the estimate
	\[\begin{aligned}
\mu^{[m]}&(z_x)|\psi^{[m]}(z_x)|\|R\varphi^{[m]}_j(z_x)\| \\
&\le \mu^{[m]}(z)\|R(\psi\cdot\varphi^{[m]}_j)(z_x)\| + \mu^{[m]}(z_x)\|R\psi(z_x)\|\|\varphi^{[m]}_j(z_x)\| \quad \forall z_{[m]} \in \BB_m, \forall x \in OS_m
	\end{aligned}  	 \]
	that
\[\lim_{\|z_x\|\to1}\mu^{[m]}(z_x)|\psi^{[m]}(z_x)|\|R\varphi^{[m]}_j(z_x)\| = 0, \]
hence,
	\begin{equation}\label{eq_5.10}
		\lim_{\|z_x\|\to1}\mu^{[m]}(z_x)|\psi^{[m]}(z_x)|\|R\varphi^{[m]}_{(k)}(z)\| = 0\quad \forall k \ge1,
	\end{equation}
	where $  \varphi^{[m]}_{(k)}(\cdot) := (\varphi^{[m]}_1(\cdot), \ldots, \varphi^{[m]}_k(\cdot)). $

 Assume that  $M^{[m]}_{R\psi,\varphi} < \infty,$  $M^{[m]}_{\psi,R\varphi} < \infty,$   $ \psi^{[m]}, \psi^{[m]}\cdot \varphi^{[m]}_j\in \mathcal B_{\mu^{[m]},0} $ for some $ m\ge 2$ and for every $ j\in \Gamma.$    By Theorem \ref{thm6_1}, $ W^0_{\psi,\varphi}: \mathcal B_{\nu,0}(B_X) \to \mathcal B_{\mu}(B_X) $ is bounded.  It suffices to show that $ W^{0,0}_{\psi,\varphi}f \in \mathcal B_{\mu,0}(B_X) $ for every $ f \in \mathcal B_{\nu,0}(B_X). $

			Let $ f \in \mathcal B_{\nu,0}(B_X) $ be arbitrarily fixed. Let $ \varepsilon >0 $ and $ k\ge1 $ be fixed. Then there exists $ r_0 \in (1/2,1) $ such that
		\begin{equation}\label{eq_6.7a} \nu(z)\|Rf(w)\| < \frac{\varepsilon}{6(M_{R\psi,\varphi}+M_{\psi,R\varphi})}, \quad  \|w\| \ge r_0. 
		\end{equation}
	Put $ \varrho := r_0 +\frac{1-r_0}{2}. $	By (\ref{eq_5.1}) we have
		\begin{equation}\label{eq_6.01}
		\begin{aligned}
			K &:= \sup_{\|w\| \le \varrho} |f(w)| <\infty, \\
				N &:=  \sup_{\|w\|\le \varrho}\|R f(w)\| =   \sup_{\|w\|\le \varrho} \dfrac{\nu(w)\|R f(w)\|}{\nu(w)} 
			\le \dfrac{\|f\|_{\mathcal B_\nu(B_X)}}{\nu(\varrho)} < \infty\quad w \in B_X.  
		\end{aligned}	
		\end{equation}
		
		Since  $\psi^{[m]}  \in \mathcal B_{\mu^{[m]},0}(\BB_m)$  and (\ref{eq_5.10}) we can find $ \theta \in (0,1)$ such that for every $ x\in OS_m $
		\begin{equation}\label{eq_6.00}
		\begin{aligned}
		\mu^{[m]}(z_x)\|R\psi^{[m]}(z_x)\| &< \dfrac{\varepsilon}{3(K+N)}, \\
		\mu^{[m]}(z_x)|\psi^{[m]}(z_x)|\|R\varphi^{[m]}_{(k)}(z)\| &<\frac{\varepsilon}{3(K+N)} \quad\text{whenever}\   \theta < \|z_x\| < 1. 
		\end{aligned}	
		\end{equation}  
		For $ \theta < \|z_x\| < 1 $ we consider two cases:
		
		$ \bullet $ The case $ \|w^{[m]}\| := \|\varphi^{[m]}_{(k)}(z_x)\| >r_0: $ Let $ \widehat w^{[m]} =  r_0\frac{w^{[m]}}{\|w^{[m]}\|}.$ We have
		\[ \begin{aligned}
			|f(\varphi^{[m]}(z_x)) - f(\widehat w^{[m]})| &=	|f(w^{[m]}) - f(\widehat w^{[m]})|  \le \int_{r_0/\|w^{[m]}\|}^1\frac{\|Rf(tw^{[m]})\|}{t}dt \\
			&\le \frac{\|w^{[m]}\|}{r_0}\int_{r_0/\|w^{[m]}\|}^1\|Rf(tw^{[m]})\|dt \\
			&\le  \frac{\varepsilon\|w^{[m]}\|}{6(M_{R\psi,\varphi}+M_{\psi,R\varphi}) r_0}\int_{r_0/\|w^{[m]}\|}^1\frac{1}{\nu(t\|w^{[m]})\|}dt \\
			&\le \frac{\varepsilon}{3(M_{R\psi,\varphi}+M_{\psi,R\varphi})}\int_{r_0}^{\|w^{[m]}\|}\frac{1}{\nu(t)}dt.
	\end{aligned} \]
		Then, by (\ref{eq_6.7a}) we have
			\[\begin{aligned}
			&\mu^{[m]}(z_x)\|R(\psi^{[m]}(f\circ\varphi^{[m]}_{(k)})(z_z))\| \\
			& \le \mu^{[m]}(z_x)\|R\psi^{[m]}(z_x)\||f(\varphi^{[m]}_{(k)}(z_x))| + \mu^{[m]}(z_x)|\psi^{[m]}(z_x)|\|\nabla_{\varphi^{[m]}_{(k)}(z)}f(\varphi^{[m]}_{(k)}(z_x))\|\|R\varphi^{[m]}_{(k)}(z_x)\| \\
			&\le \mu^{[m]}(z_x)|R\psi^{[m]}(z_x)||f(\varphi^{[m]}_{(k)}(z_x)) - f(\widehat w^{[m]})| + \mu^{[m]}(z_x)\|R\psi^{[m]}(z_x)\||f(\widehat w^{[m]})| \\
			&\quad + \mu^{[m]}(z_x)|\psi^{[m]}(z_x)|\|\nabla_{\varphi^{[m]}_{(k)}(z)}f(\varphi^{[m]}_{(k)}(z_x))\|\|R\varphi^{[m]}_{(k)}(z_x)\| \\
			&\le  M^{[m]}_{R\psi,\varphi}\frac{\varepsilon}{3(M_{R\psi,\varphi}+M_{\psi,R\varphi})} + K\frac{\varepsilon}{3(K+N)} \\
			&\quad + \sup_{\|\varphi^{[m]}_{(k)}(w)\|>r_0}\nu(\varphi^{[m]}_{(k)}(w))\|\nabla_{\varphi^{[m]}_{(k)}(w)}f(\varphi^{[m]}_{(k)}(w))\| \sup_{w\in \BB_n}\dfrac{\mu^{[m]}(w)|\psi^{[m]}(w)|\|R\varphi^{[m]}_{(k)}(w)\|}{\nu(\varphi^{[m]}_{(k)}(w))} \\
			&= 	\dfrac{2\varepsilon}{3}	+  M^{[m]}_{\psi,R\varphi}\frac{\varepsilon}{6(M_{R\psi,\varphi}+M_{\psi,R\varphi})} <\varepsilon.
		\end{aligned}\]
	
	$ \bullet $ The case $ \|\varphi_{(m)}(z)\| \le r_0: $  By (\ref{eq_6.01}) we have
		\[ 	\begin{aligned}
		\mu^{[m]}(z_x)\|R(\psi^{[m]}&(f\circ\varphi^{[m]}_{(k)})(z_x))\| \\
		&\le \mu^{[m]}(z_x) |\psi^{[m]}(z_x)|\|R\varphi^{[m]}_{(k)}(z_x)\|\|\nabla_{\varphi^{[m]}_{(k)}(z)}f(\varphi^{[m]}_{(k)}(z_x))\| \\
		&\quad + \mu^{[m]}(z_x)\|R\psi^{[m]}(z_x)\| |f(\varphi^{[m]}_{(k)}(z_x))| \\
		&< N\frac{\varepsilon}{3(K+N)}+ K\frac{\varepsilon}{3(K+N)} <\varepsilon.
	\end{aligned}  \]
Now,  since $ \varphi^{[m]}_{(k)}(z_x) \to \varphi^{[m]}(z_x) $ as $ k\to \infty, $ combining 
two cases, we obtain
\[ \mu^{[m]}(z_x)\|R(\psi^{[m]}(f\circ\varphi^{[m]}(z_x))\|  <\varepsilon \quad\forall x\in OS_m,\ \theta < \|z_x\| < 1.\]
Finally, fix $ z \in B_X, $ $ \theta < \|z\|  <1.  $ Consider   $ x =(\frac{z}{\|z\|},x_2, \ldots, x_m) \in OS_m $ and put $ z_{[m]} := (\|z\|, 0,\ldots,0) \in \C^m. $ Then $ \theta < \|z_x\| = \|z_{[m]}\| = \|z\| < 1$  and by as (\ref{norm_z_x}) we obtain
Consequently,
\[ \mu(z)\|R(\psi(f\circ\varphi(z))\|  = \mu^{[m]}(z_x)\|R(\psi^{[m]}(f\circ\varphi^{[m]}(z_x))\| <\varepsilon.\]
hence,     $ W^{0,0}_{\psi,\varphi}(f) \in \mathcal B_{\mu,0}(B_X). $ 

\medskip
(3) $\Rightarrow$ (5): Use  an argument analogous and an estimate to the previous one.

	(5) $ \Rightarrow $ (4): 
	Assume that $ W^{0,0}_{\psi,\varphi} $ is bounded. As in the proof of (6) $ \Rightarrow $ (4) in Theorem \ref{thm6_1} we obtain that  $ \psi \in \mathcal B_{\mu}(B_X) $ and the estimates $M_{R\psi,\varphi} < \infty,$  $M_{\psi,R\varphi} < \infty$  because the test function $ \beta_{w} $ given by (\ref{test_func2}) belongs to $ \mathcal B_{\nu,0}(B_X). $
	
	It remains to check that $\psi, \psi \cdot\varphi_j \in \mathcal B_{\mu,0}(B_X)$ for every $ j\in \Gamma. $ It is clear because it is easy to see that $ \psi = W^{0,0}_{\psi,\varphi}(1), $ $\psi\cdot\varphi_j = W^{0,0}_{\psi,\varphi}(f_j), $ where $ f_j, 1 \in \mathcal B_{\nu,0}(B_X)$ given by $f_j(z) := z_j $  for every $ z \in B_X. $ 
	 
\end{proof}

Now we investigate the compactness of weighted composition operators $ W_{\psi,\varphi}. $
\begin{thm}\label{thm_compact} Let $\psi \in H(B_X), $  $\varphi \in S(B_X)$ and $ \mu, \nu $ be   normal weights on $ B_X$ such that $ \int_0^1\frac{dt}{\nu(t)} = \infty. $ Then  
		\begin{enumerate}
		\item [\rm (A)] The  following are equivalent:
		\begin{enumerate}
			\item [\rm (1)]	$ (\ref{B_3a}) \to 0, $ $ (\ref{B_3b})\to 0 $ as $ \|\varphi_{(k)}(y)\|\to1 $   for every $k\ge1, $
						\item[\rm (2)] $W_{\psi, \varphi}: \mathcal B_{\nu}(B_X) \to \mathcal B_{\mu}(B_X)$     is compact; 
			\item [\rm (3)]$W^0_{\psi, \varphi}: \mathcal B_{\nu, 0}(B_X) \to \mathcal B_{\mu}(B_X)$     is compact.
		\end{enumerate}
		\item[\rm (B)] Under the additional assumption   that  there exists $ m\ge 2 $ such that     
		\begin{equation}\label{phi_relative_compact} 
			\begin{aligned}
				B[\varphi^{[m]}, r] &:=\big\{\varphi^{[m]}(y): \|\varphi^{[m]}(y)\| < r,  y \in \BB_m \big\}\\
				&\quad\ \ \text{is relatively compact for every $ 0\le r<1,$ }
			\end{aligned} 
		\end{equation}
		the assertions (2), (3) and  following are equivalent:
				\begin{enumerate}
			\item[\rm(4)]  $ (\ref{B_1a}) \to 0, $ $ (\ref{B_1b})\to 0 $ as $ \|\varphi^{[m]}(y)\|\to1; $ 
				\item[\rm(5)]  $ (\ref{B_2a}) \to 0, $ $ (\ref{B_2b})\to 0 $ as $ \|\varphi^{[F]}(y)\|\to1 $  for every $ F \subset \Gamma $ finite;
					\item[\rm(6)]  $ (\ref{B_4a}) \to 0, $ $ (\ref{B_4b})\to 0 $ as $ \|\varphi(y)\|\to1. $ 
			\end{enumerate}
	\end{enumerate}
	\end{thm}
\begin{proof} First we prove (B). 
	
	\medspace
	(B) 	The implications (6) $\Rightarrow$  (5)  $\Rightarrow$  (4) and   (2) $\Rightarrow$  (3)  are  obvious. 
	
\medspace
(4) $\Rightarrow$ (2): For each $ x \in OS_m $ we write $ z_x := \sum_{k=1}^mz_kx_k. $  It follows from the assumption (4), Theorem \ref{thm6_1} and the condition (\ref{phi_relative_compact}) that   $W_{\psi, \varphi}$   
  is bounded and $ \psi \in \mathcal B_\mu(B_X). $   

	Let $\{f_n\}_{n\ge1}$ be a bounded sequence in $\mathcal B_{\nu}(B_X)$ converging to $0$ uniformly on compact subsets of $B_X$ and fix an $\varepsilon > 0.$ 
By the hypothesis on the sequence $ \{f_n\}_{n\ge1}, $ 
there exists a positive integer $ n_0 $ such that
\[	 		 |f_n(w)| \le \frac{\varepsilon}{3K}, \quad \forall  n \ge n_0,\  \forall w \in \overline{B[\varphi^{[m]},r_0]},\]
where $K = \|\psi\|_{\mathcal B_\mu(B_X)} +\sup_{n \in \N}\|f_n\|_{\mathcal B_\nu(B_X)}.$ 

Denote $ w^{[m]} = \varphi^{[m]}(z_x).$	 Since $ \int_0^1\frac{dt}{\nu(t)} = \infty $ and $ (\ref{B_1a}) \to 0, $ $ (\ref{B_1b})\to 0 $ as $ \|\varphi^{[m]}(y)\|\to1 $, $ \psi^{[m]} \in \mathcal B_{\mu,0}(\BB_m) $ and
for any $\varepsilon > 0,$ there exists $ r_0 \in (1/2, 1) $ 	such that whenever $  r_0 < \|w^{[m]}\| < 1 $
\begin{equation}\label{limit_2compact_a}
	\begin{aligned}
& \mu^{[m]}(z_x)\|R\psi^{[m]}(z_x)\|\int_{r_0}^{\|w^{[m]}\|}\frac{dt}{\nu^{[m]}(t)} <\frac{\varepsilon}{6/C_2}\\
& \frac{\mu^{[m]}(z_x)|\psi^{[m]}(z_x)|\|R\varphi^{[m]}(z_x)\|}{\nu(\varphi^{[m]}(z_x))} <\frac{\varepsilon}{3K}.
	\end{aligned} 	 \end{equation}
	
	Now, for every $ n\ge n_0 $ and $ r_0 < \|w^{[m]}\| <1, $ with noting that  $$ \widehat w^{[m]} := r_0\frac{w^{[m]}}{\|w^{[m]}\|} \in   \overline{B[\varphi^{[m]},r_0]},$$ by (\ref{test_g2}) and (\ref{limit_2compact_a}), as in the proof (1) $ \Rightarrow $ (5) in Theorem \ref{thm6_1}   we have
	
		\begin{equation}\label{est_compact1}
		\aligned
		\mu^{[m]}(z_x)&\|R (W_{\psi,\varphi}f_n)(z_x)\|   \\
		&\le   \mu^{[m]}(z_x)\|R\psi^{[m]}(z_x)||f_n(\varphi^{[m]}(z_x))\| \\
		&\quad+  \mu^{[m]}(z_x) |\psi^{[m]}(z_x)|\|\nabla_{\varphi^{[m]}(z_x)}f_n(\varphi^{[m]}(z_x))\| \|R\varphi^{[m]}(z_x)\|  \\
		&\le \mu^{[m]}(z_x)\|R\psi^{[m]}(z_x)\| |f_n(\varphi^{[m]}(z_x)) - f_n(\widehat w^{[m]})| \\
		&\quad + \mu^{[m]}(z_x)\|R\psi^{[m]}(z_x)\| |f_n(\widehat w^{[m]})| \\
			&\quad+  \mu^{[m]}(z_x) |\psi^{[m]}(z_x)|\|\nabla_{\varphi^{[m]}(z_x)}f_n(\varphi^{[m]}(z_x))\| \|R\varphi^{[m]}(z_x)\|  \\
		&\le \mu^{[m]}(z_x)\|R\psi^{[m]}(z_x)\|\frac{\|w^{[m]}\|}{r_0}\int_{r_0/\|w^{[m]}\|}^1\|Rf(tw^{[m]})\|dt \\
		&\quad + \|\psi\|_{\mathcal B_\mu(B_X)}\frac{\varepsilon}{3K} + \frac{\varepsilon}{3K}\|f_n\|_{\mathcal B_{\nu}(B_X)} \\
		&\le \mu^{[m]}(z_x)\|R\psi^{[m]}(z_x)\|\frac{2}{C_2}\int_{r_0/\|w^{[m]}\|}^1\frac{1}{\nu(t\|w^{[m]}\|)}dt + \frac{2\varepsilon}{3}\\
&\le \mu^{[m]}(z_x)\|R\psi^{[m]}(z_x)\|\frac{2}{C_2}\int_{r_0}^{\|w^{[m]}\|}\frac{dt}{\nu^{[m]}(t)} + \frac{2\varepsilon}{3}\\
&<\frac{\varepsilon}{6/C_2}\frac{2}{C_2}+\frac{2\varepsilon}{3} =\varepsilon.
			\endaligned
	\end{equation}

	On the other hand, since $\{f_n\}_{n\ge1}$  converges to $0$ uniformly on compact subsets of $B_X,$ by Cauchy integral formula and (\ref{phi_relative_compact} ),  it is clear that
	$$\sup_	{y \in\overline{B[\varphi^{[m]}, r]}}\|\nabla_yf_n(y)\| \to 0\quad \text{as}\  n\to \infty.$$ 
	Then, in the case $ \|w^{[m]}\| \le r_0 $ with the estimate as above we obtain
	\begin{equation}\label{est_compact2}
	\mu^{[m]}(z_x)\|R (W_{\psi^{[m]},\varphi^{[m]}}f_n)(z_x)\| <\varepsilon  \quad \text{for every} \ n\ge n_0.
		\end{equation}
Finally, with the  note that the estimates (\ref{est_compact1}), (\ref{est_compact2})  are independent of $ x \in OS_m, $  we obtain 
\[  \|W_{\psi, \varphi}(f_n)\|_{\mathcal B_\mu(B_X)} = \sup_{x\in OS_m} \|W_{\psi^{[m]}, \varphi^{[m]}}((f_n)_x)\|_{\mathcal B_\mu(\BB_m)} \to 0 \quad \text{as} \ n \to \infty.\]
Hence, it implies from Proposition \ref{prop_W_compact} that $W_{\psi, \varphi}$ is compact.
	
	\medskip
	(3) $\Rightarrow$ (6): 
	Suppose $W^0_{\psi, \varphi}$ is compact. Then  clearly, $W^0_{\psi, \varphi}$ is bounded.

	Firstly, assume that  $ (\ref{B_4b})\not\to 0 $ as $ \|\varphi(z)\|\to1, $ $ z\in B_X. $ Then we can take $ \varepsilon_0 > 0 $ and a sequence $\{z^n\}_{n\ge1}$ in $ B_X $ such that
	$\|w^n\| := \|\varphi(z^n)\| \to 1$ but
	\[  \frac{\mu(z^n)|\psi(z^n)|\|R\varphi(z^n)\|}{\nu(\varphi(z^n))} \ge \varepsilon_0	\quad \text{for every} \  n =1, 2, \ldots   \]
	Consider the   sequence $ \{\gamma_{w^n}\}_{n\ge1} $ defined  by (\ref{test_func3}). By Proposition \ref{lem3}, this sequence is bounded in $\mathcal B_{\nu}(B_X)$ and converges to $0$ uniformly on compact subsets of $ B_X. $  By an argument  similar to the proof    (6) $ \Rightarrow $ (4) of Theorem \ref{thm6_1} we get the one as the estimate  (\ref{est_contrad_1}): 
	\[ \begin{aligned}
		\|W^0_{\psi,\varphi}&(\gamma_{w^n})\|_{\mathcal B_\mu(B_X)} \\
		&\ge \|\mu(z^n)R\psi(z^n)\gamma_{w^n}(\varphi(z^n)) + \mu(z^n)\psi(z^n)\nabla\gamma_{w^n}(\varphi(z^n))R\varphi(z^n)\| \\
		&\ge \|\mu(z^n)\psi(z^n)\nabla\gamma_{w^n}(\varphi(z^n))R\varphi(z^n)\| - \|\mu(z^n)R\psi(z^n)\gamma_{w^n}(\varphi(z^n))\| \\
		&= \mu(z^n)|\psi(z^n)| \|R\varphi(z^n)\|\dfrac{2\int_0^{\|\varphi(z^n)\|^2}g(t)dt}{\int_0^{\|\varphi(z^n)\|^2}g(t)dt}\|\nabla\gamma_{w^n}(\varphi(z^n))\| \\
		&\quad - \|\mu(z^n)R\psi(z^n)\gamma_{w^n}(\varphi(z^n))\| \\
		&\ge 2\dfrac{\mu(z^n)|\psi(z^n)| \|R\varphi(z^n)\|}{\nu(|\varphi(z^n)|)}\nu(\|\varphi(z^n)\|^2)g(\|\varphi(z^n)\|^2)\dfrac{\nu(\|w^n\|)}{\nu(\|w^n\|^2)} \\
		&\quad - \mu(z^n)\|R\psi(z^n)\gamma_{w^n}(\varphi(z^n))\| \\
		&\ge 2\varepsilon_0 C_1C_\mu - \mu(z^n)\|R\psi(z^n)\gamma_{w^n}(\varphi(z^n))\|. 
	\end{aligned} 
	\]
	This implies that
	\[ \begin{aligned}
		0<  2\varepsilon_0 C_1C_\mu &\le   \mu(z^n)\|R\psi(z^n)\gamma_{w^n}(\varphi(z^n))\| + \|W^0_{\psi,\varphi}(\gamma_{w^n})\|_{\mathcal  B_\mu(B_X)}  \\
				&= \|\psi\|_{\mathcal B_\mu(B_X)}|\gamma_{w^n}(\varphi(z^n))| +  \|W^0_{\psi,\varphi}(\gamma_{w^n})\|_{\mathcal  B_\mu(B_X)} \\
		&\to 0 \quad\text{as}\ n \to \infty.
	\end{aligned}\]
	This is impossible.
	
	Finally, let $ \{z^n\}_{n\ge1} $ be a sequence in $B_X $ such that $ \|w^n\| := \|\varphi(z^n)\| \to 1  $ as $ n\to \infty. $ Consider the sequence $ \{\gamma_{w^n}\}_{n\ge 1} $     as the above.  By (\ref{es_gamma_w1}), as (\ref{es_gamma_w2}), we have
\[	\begin{aligned}
		 \mu(z^n)&\|R\psi(z^n)\|\int_0^{\|w^n\|}\dfrac{dt}{\nu(t)} \\
		 & \le  \mu(z^n)\|RW^0_{\psi, \varphi}\gamma_{w^n}(z^n)\|  + \frac{\mu(z^n)|\psi(z^n)|\|R\varphi(z^n)\|}{\nu(\varphi(z^n))}\|\gamma_{w^n}\|_{\mathcal B_\nu(B_X)} 
		 \end{aligned}\]
	
%
	
	Then, by Theorem \ref{thm6_1} and Proposition \ref{prop_W_compact}, as (\ref{est_compact1}), with note that  $ (\ref{B_4b})\to 0 $ as $ \|\varphi(z)\|\to1, $  we obtain
 $ (\ref{B_4a})\to 0 $ as $ \|\varphi(z)\|\to1, $ and the proof of (B) is complete. 
	
	 Finally, we prove (A).
	
	\medspace
	(A)  Since (2) $ \Leftrightarrow $ (3) $ \Leftrightarrow $ (6) and (6) $ \Rightarrow $ (1) is obvious, it suffices to prove (1) $ \Rightarrow $ (2).
	We obtain this proof  by an estimate analogous to that used for the proof of (4) $ \Rightarrow $ (2) by using the known fact that 
	\begin{equation}\label{phi_relative_compact_1} 
		\begin{aligned}
			B[\varphi_{(k)},r_0] &= \{\varphi_{(m)}(z): \ \|\varphi_{(k)}(z)\| \le r_0 \} \subset \BB_k \subset \C^k\\
			&\quad\ \ \text{is relatively compact for every $ 0\le r_0<1 $ and $ k\ge1, $ }
		\end{aligned} 
	\end{equation}
	instead of (\ref{phi_relative_compact}).
\end{proof}
 \begin{rmk}\label{rmk_6.1}  \begin{itemize}
 		\item[\rm(a)]  The implication (3) $ \Rightarrow $ (1) was proved without appeal to the assumption (\ref{phi_relative_compact}). 
 		\item[\rm(b)] In general, the assumption (2) or (3)    does not imply   (\ref{phi_relative_compact}). This means that if (\ref{phi_relative_compact}) was moved into the assumption (1) then the equivalence of the statements in Theorem \ref{thm_compact}   is broken.
 		An illustrative example of this comment will be introduced after the following theorem:
 		 	\end{itemize}
 	 	    	   \end{rmk}
  \begin{thm}\label{thm_6.4}
Let $\psi \in H(B_X)$  and $\varphi \in S(B_X).$ Assume that $ W_{\psi, \varphi}: \mathcal B_\nu(B_X) \to \mathcal B_\mu(B_X) $ is compact. Then
	\begin{equation}\label{compact_3}   \varphi(r B_X)  \ \text{is relatively compact for every $ 0\le r < 1$}.
\end{equation}
  \end{thm}
   \begin{proof}
  	For every $ z\in B_X, $ consider  the function $ \delta_z $ given by $ \delta_z(f) = f(z)$ for every $ f \in \mathcal B_\mu(B_X). $ 
  By (\ref{eq_5.1}), it is clear that $ \delta_z \in (\mathcal B_\mu(B_X))'.$ Moreover, we have
  \begin{equation}\label{eq_6.18}
  	\dfrac{1}{2}\|z -w\| \le \|\delta_z - \delta_w\| \quad \forall z,w \in B_X.
  \end{equation} 
  Indeed, it is easy to check by direct calculation that
  \[ \dfrac{1}{2}\|z-w\| \le \sqrt{1-\dfrac{(1-\|z\|^2)(1-\|w\|^2)}{|1-\langle z,w\rangle|^2}} = \varrho_X(z,w) \]
  where $ \varrho_X $ is the pseudohyperbolic metric in $ B_X $ (see \cite[p.99]{GR}). On the other hand, we also have
  \[ \varrho_X(z,w) = \sup\{\varrho(f(z),f(w)):\ \text{$ f \in H^\infty(B_X) $ with $ \|f\|_\infty \le 1 $}\} \]
  (see (3.4) in \cite{BGM}), where $ \varrho(x,y) = \big|\frac{x-y}{1-\overline{x}y}\big|, $  $ x,y \in \BB_1, $ is the pseudohyperbolic metric in $ \BB_1. $ Note that, since the function $ \eta \mapsto \frac{\eta}{1-\overline{f(z)}f(w)} $ is holomorphic from $ \BB_1 $ into $ \BB_1 $ and $ f(z) - f(w) \mapsto 0,$ it follows from Schwarz's lemma that 
  $ \varrho(f(z), f(w)) \le |f(z) - f(w)| $ for every $ z, w \in B_X. $ Consequently,
  \[\begin{aligned}
  	\varrho_X(z,w) &\le \sup\{|f(z) - f(w)|:\ \ \text{for $ f \in H^\infty(B_X) $ with $ \|f\|_\infty \le 1 $}\}  \\
  	&\le \sup\{|\delta_z(f) - \delta_w(f)|:\ \ \text{for $ f \in H^\infty(B_X) $ with $ \|f\|_\infty \le 1 $}\} \\
  	&= \|\delta_z- \delta_w\|.
  \end{aligned}\]
  Hence, (\ref{eq_6.18}) is proved.

  For $ 0<r<1, $ the set $ V_r := \{\delta_z:\ \|z\| \le r\} \subset (\mathcal B^\mu_\nabla(B_X))'$ is bounded. Then, by the compactness of 
  $W_{\psi, \varphi} $ the set
  \[  (W_{\psi,\varphi})^*(V_z) = \{\psi(z)\delta_{\varphi(z)}:\ \|z\|\le r\} \] is relatively compact in $ (H^\infty(B_X))'. $

  It should be noted that, for every  subset  $ K $ of the dual of  a Banach space $ E  $ and every bounded subset $ D \subset \C, $ if the  set $ \{t\eta: \ t \in D, \eta \in A\} $ is relatively compact in $ E  $  then $ A \subset E' $ is relatively compact. With this fact in hand, since the  set $ \{\psi(z):\ \|z\|\le r\} $ is bounded,    the set $  \{\delta_z, \ \|z\|\le r\} $ is relatively compact. Then, it follows from the   inequality (\ref{eq_6.18}) 
  that $ \varphi(rB_X)$  is relatively compact.
   \end{proof}

       We introduce  below is an example which shows that (\ref{compact_3}) does not imply  (\ref{phi_relative_compact}).
 
	\begin{exam} Let $ \{e_j\}_{j\ge 1} $ be an orthonormal sequence in a Hilbert space $ X. $   Consider the function $ \varphi \in S(B_X) $ given by
		\[ \varphi(z) := \sum_{n=1 }^\infty\langle z,e_n\rangle^ne_n  \quad \forall z  \in B_X. \]
		It is easy to check that $ \varphi(rB_X) $ is relatively compact for every $ 0<r<1. $ 
		
		Now we show that  $ B[\varphi,\frac{1}{2}]  $ is not relatively compact.
			Consider the sequence $ \{z_k\}_{k\ge1} \subset B_X$ given by
		\[ z_k =  \frac1{\sqrt[k]{4}}e_k\quad \forall k \ge1. \]
		It is obvious $ \|\varphi(z_k)\|<\frac12 $ for every $ k\ge1. $ Then for every $ k\ge 1 $ and $ s>1 $ we have
		\[ \|\varphi(z_k) - \varphi(z_{k+s})\| =  \frac{\sqrt{2}}{4}.\]
Thus, we get the desired claim.
	\end{exam}

\begin{rmk} Under the additional condition that $ \varphi(0) = 0,$ the limits  $ (\ref{B_1a}) \to 0, $ $ (\ref{B_1b})\to 0 $ in Theorem \ref{thm_compact}  hold as  $ \|z_{[m]}\| \to 1.$ 
  Indeed, in a more general framework, it suffices to  show $ \|\varphi(z)\|\le \|z\| $ for every $ z\in B_X. $ That means we  have to give  an infinite version of Schwarz's lemma. 
  
  For each $ z \in X, z \neq 0 $ and $ w \in \overline{B_X}, $ applying classical Schwarz's lemma to the functions $ \phi_{z,w}: \BB_1 \to \BB_1$ given by
  \[ \phi_{z,w}(t) := \langle\varphi(tz/\|z\|), w\rangle \quad\forall t \in \BB_1, \]
  we have 
  \[ |\phi_{z,w}(t)| \le |t|. \]
  Then, choosing $ t = \|z\| $ and $ w = \frac{\overline{\varphi(z)}}{\|\varphi(z)\|} $ we get the desired inequality.
  \end{rmk}

%
%
%

\begin{cor}\label{cor_6.7}
	Assume that $\sup_{z\in X}\|\varphi(z)\| <\infty.$  Then  $ W_{\psi,\varphi}, $ $W^0_{\psi,\varphi} $ are compact if and only if $ \varphi(B_X) $ is relatively compact.
\end{cor}
Indeed, by the hypotheses,   (\ref{phi_relative_compact}),   (\ref{compact_3})  and the assumption (6) in Theorem \ref{thm_compact} always hold. 

\begin{thm}\label{thm_compact2} Let $\psi \in H(B_X),$  $\varphi \in S(B_X)$  and $ \mu, \nu $ be   normal weights on $ B_X $ such that $ \int_0^1\frac{dt}{\nu(t)} < \infty. $ Then
		\begin{enumerate}
		\item [\rm (A)] The  following are equivalent:
		\begin{enumerate}
			\item [\rm (1)]	 $ \psi \in \mathcal B_\mu(B_X) $ and   $ (\ref{B_3b})\to 0 $ as $ \|\varphi_{(k)}(y)\|\to1 $   for every $ k\ge1; $
			\item[\rm (2)] $W_{\psi, \varphi}: \mathcal B_{\nu}(B_X) \to \mathcal B_{\mu}(B_X)$     is compact; 
			\item [\rm (3)]$W^0_{\psi, \varphi}: \mathcal B_{\nu, 0}(B_X) \to \mathcal B_{\mu}(B_X)$     is compact.
		\end{enumerate}
		\item[\rm (B)] Under the additional assumption   that  there exists $ m\ge 2 $ such that     
		(\ref{phi_relative_compact}) holds, the assertions (2), (3) and  following are equivalent:
		\begin{enumerate}
			\item[\rm(4)]  $ \psi^{[m]} \in \mathcal B_{\mu^{[m]}}(\BB_m) $ and  $ (\ref{B_1b})\to 0 $ as $ \|\varphi^{[m]}(y)\|\to1; $ 
			\item[\rm(5)]  $ \psi^{[F]} \in \mathcal B_{\mu^{[F]}}(\BB_{[F]}) $ and $ (\ref{B_2b})\to 0 $ as $ \|\varphi^{[F]}(y)\|\to1 $  for every $ F \subset \Gamma $ finite;
			\item[\rm(6)]  $ \psi \in \mathcal B_\mu(B_X) $ and  $ (\ref{B_4b})\to 0 $ as $ \|\varphi(y)\|\to1. $ 
				\end{enumerate}
	\end{enumerate}
\end{thm}
\begin{proof}
	As Theorem \ref{thm_compact}, it suffices to prove (4) $ \Rightarrow $ (2) and (3) $\Rightarrow$ (6).
	
	(4) $ \Rightarrow $ (2): By Theorem \ref{thm_compact} it suffices to show that $ (\ref{B_1a}) \to0 $ as $ \|\varphi^{[m]}(y)\|\to 1, $ $ y\in \BB_m. $ It is easy to prove this fact from the assumptions that  $ \psi^{[m]} \in \mathcal B_{\mu^{[m]}}(\BB_m) $ and $ \int_0^1\frac{dt}{\nu(t)}<\infty. $
	
%
%
%
		
	(3) $ \Rightarrow $ (6): Since $ W^0_{\psi\varphi} $ is bounded, as in Theorem \ref{thm6_1}, we have $ \psi \in \mathcal B_\mu(B_X). $  
	
	Now, we show that $ (\ref{B_4b})\to 0 $ as $ \|\varphi(y)\|\to1 $ for $ y \in B_X.$
	Otherwise, we would have a sequence  $ \{z^n\}_{n\ge1} \subset B_X $ and
	some positive constant $ \varepsilon_0 >0 $ such that such that $ \|w^n\| := \|\varphi(z^n)\| \to 1 $ but
	\[ \frac{\mu(z^n)|\psi(z^n)| \|R\varphi(z^n)\|}{\nu(\varphi(z^n))} \ge \varepsilon_0.  \]
	
	We may assume that
	\begin{equation}\label{eq_wtend1}
		1-\dfrac{1}{n^2} < \|w^n\| < 1 \quad\text{and}\quad \lim_{n\to\infty}\|w^n\| = \|w^n\| = 1.
	\end{equation} 
	Consider the sequence $ \{\theta_{w^n}\}_{n\ge1} $ which is bounded in $ \mathcal B_\nu(B_X) $ and converges to $ 0 $ uniformly on compact subsets of $ B_X$ (see Proposition \ref{lem3}). Then
	\[ \begin{aligned}
		\|W^0_{\psi,\varphi}&\theta_{w^n}\|_{\mathcal B_\mu(B_X)} +\sup_{z\in B_X} \mu(z)\|R\psi(z)\||\theta_{w^n}(\varphi(z))| \\
		&\ge \sup_{z\in \BB_n}\mu(z)|\psi(z)| \|R\eta_{w^n}(\varphi(z))\| \|R\varphi(z)\|\\
		&\ge \mu(z^n)|\psi(z^n)| \|R\varphi(z^n)\||g(\|w^n\|^{n+2})-g(\|w^n\|^2)|\\
		&\ge \frac{\mu(z^n)|\psi(z^n)| \|R\varphi(z^n)\|}{\nu(\varphi(z^n)}\nu(\varphi(z^n)|g(\|w^n\|^{n+2})-g(\|w^n\|^2)| \\
		&\ge \varepsilon_0 \nu(\varphi(z^n) \big|g(\|w^n\|^{n+2})-g(\|w^n\|^2)\big| \\
		&\ge  \varepsilon_0\bigg[\nu(\|w^n\|^2)g(\|w^n\|^2)\frac{\nu(\|w^n\|)}{\nu(\|w^n\|^2)} - \nu(\|w^n\|^{n+2})g(\|w^n\|^{n+2})\frac{\nu(\|w^n\|)}{\nu(\|w^n\|^{n+2})}\bigg].
	\end{aligned} \]
	On the other hand, it follows from (\ref{eq_wtend1}) that $ \|w^n\|^{n+2} \to 1 $ as $ n\to \infty. $ Then
	\[ \begin{aligned}
		0 &\le \limsup_{n\to\infty}\frac{\nu(w^n)}{\nu(\|w^n\|^{n+2}} \\
		&= \limsup_{n\to\infty}\frac{\nu(w^n)/(1-\|w^n\|^2)^a(1+\|w^n\| + \|w^n\|^2+\cdots+\|w^n\|^{n+1})^a}{\nu(\|w^n\|^{n+2})/(1-\|w^n\|^{n+2})^a} \\
		&\le \limsup_{n\to\infty}\frac{1}{(1+\|w^n\| + |w^n\|^2+\cdots+\|w^n\|^{n+1})^a} \\
		&\le \limsup_{n\to\infty}\frac{1}{(2+n)^a} =0.
	\end{aligned} \]
	Therefore, it follows from (\ref{eq_frac_nu}) that
	\[ \begin{aligned}
		&\liminf_{n\to\infty}\Big[\|W^0_{\psi,\varphi}\theta_{w^n}\|_{\mathcal B\mu(B_X)} + \mu(z)\|R\psi(z)\| |\theta_{w^n}(\varphi(z))|\Big] \\
		&\ge \varepsilon_0\bigg[\liminf_{n\to\infty}\nu(\|w^n\|^2)g(\|w^n\|^2)\frac{\nu(\|w^n\|)}{\nu(\|w^n\|^2)} -\limsup_{n\to\infty} \nu(\|w^n\|^{n+2})g(\|w^n\|^{n+2})\frac{\nu(\|w^n\|)}{\nu(\|w^n\|^{n+2})}\bigg] \\
		&\ge \varepsilon_0 \big[\inf_{t\in [0,1)}\nu(t)g(t) - 0\big] > 0.
	\end{aligned} \]
	This contradicts $ \|W^0_{\psi,\varphi}\theta_{w^n}\|_{\mathcal B\mu(B_X)} \to0 $ and $ \theta_{w^n}(\varphi^{[m]}(z^n)) \to 0 $ as $ n\to\infty. $
	
%
	This  concludes the proof.
\end{proof}
In the same way as in  the proof of Theorems \ref{thm_compact}, \ref{thm_compact2}, by using Theorem \ref{thm2_2}   we obtain the following results on the compactness of the operator $W^{0,0}_{\psi, \varphi}: \mathcal B_{\nu,0}(B_X) \to \mathcal B_{\mu,0}(B_X).$ 
\begin{thm}\label{thm2_2_cp} Let $\psi \in H(B_X),$  $\varphi   \in S(B_X).$    and  $ \mu, \nu $ be  normal weights on $ B_X $  such that    $ \int_0^1\frac{dt}{\nu(t)} = \infty. $ Then
	\begin{enumerate}
		\item [\rm(A)] The following are equivalent:
		\begin{enumerate}
			\item   [\rm (1)]	$\psi, \psi \cdot\varphi_j \in \mathcal B_{\mu,0}(B_X),$ $ (\ref{B_3a}) \to 0, $ $ (\ref{B_3b})\to 0 $ as $ \|\varphi_{(k)}(y)\|\to1 $    for every $ j\in \Gamma $ and for every $ k \ge 1;$
			\item[\rm(2)]  $W^{0,0}_{\psi, \varphi}: \mathcal B_{\nu,0}(B_X) \to \mathcal B_{\mu,0}(B_X)$     is compact,	
		\end{enumerate}
	\item[\rm (B)] Under the additional assumption   that  there exists $ m\ge 2 $ such that     
	(\ref{phi_relative_compact}) holds, the assertion (2)  and  following are equivalent:
	\begin{enumerate}
		\item[\rm(3)]   $ \psi^{[m]}, \psi^{[m]}\cdot\varphi^{[m]}_j \in \mathcal B_{\mu^{[m]}}(\BB_m),$	$ (\ref{B_1a}) \to 0, $ $ (\ref{B_1b})\to 0 $ as $ \|\varphi^{[m]}(y)\|\to1 $ for every $ j \in \Gamma; $
		\item[\rm(4)] $ \psi^{[F]}, \psi^{[F]}\cdot\varphi^{[F]}_j \in \mathcal B_{\mu^{[F]}}(\BB_{[F]}), $   $ (\ref{B_2a}) \to 0, $ $ (\ref{B_2b})\to 0 $ as $ \|\varphi^{[F]}(y)\|\to1 $    for every $ F \subset \Gamma $ finite and for every $ j \in \Gamma; $ 
		\item[\rm(5)] $ \psi, \psi\cdot\varphi_j \in \mathcal B_{\mu}(B_X), $ $ (\ref{B_4a}) \to 0, $ $ (\ref{B_4b})\to 0 $ as $ \|\varphi(y)\|\to1.$   
	\end{enumerate}
	\end{enumerate}
 \end{thm}
\begin{thm}\label{thm2_2_cp1} Let $\psi \in H(B_X),$  $\varphi   \in S(B_X).$    and  $ \mu, \nu $ be   normal weights on $ B_X $  such that    $ \int_0^1\frac{dt}{\nu(t)} < \infty. $ Then
	\begin{enumerate}
		\item [\rm(A)] The following are equivalent:
		\begin{enumerate}
			\item   [\rm (1)]	 $\psi, \psi \cdot\varphi_j \in \mathcal B_{\mu,0}(B_X),$ $ (\ref{B_3b})\to 0 $ as $ \|\varphi_{(k)}(y)\|\to1 $   for every $ j\in \Gamma $ and for every $ k \ge 1;$
			\item[\rm(2)]  $W^{0,0}_{\psi, \varphi}: \mathcal B_{\nu,0}(B_X) \to \mathcal B_{\mu,0}(B_X)$     is compact, 	
		\end{enumerate}
		\item[\rm (B)] Under the additional assumption   that  there exists $ m\ge 2 $ such that     
		(\ref{phi_relative_compact}) holds, the assertion (2)  and  following are equivalent:
		\begin{enumerate}
			\item[\rm(3)]  	$ \psi^{[m]}, \psi^{[m]}\cdot\varphi^{[m]}_j \in \mathcal B_{\mu^{[m]}}(\BB_m), $ $ (\ref{B_1b})\to 0 $ as $ \|\varphi^{[m]}(y)\|\to1 $    for every $ j \in \Gamma; $
			\item[\rm(4)] $ \psi^{[F]}, \psi^{[F]}\cdot\varphi^{[F]}_j \in \mathcal B_{\mu^{[F]}}(\BB_{[F]}), $   $ (\ref{B_2b})\to 0 $ as $ \|\varphi^{[F]}(y)\|\to1 $     for every $ F \subset \Gamma $ finite and for every $ j \in \Gamma; $ 
			\item[\rm(5)] $ \psi, \psi\cdot\varphi_j \in \mathcal B_{\mu}(B_X), $   $ (\ref{B_4b})\to 0 $ as $ \|\varphi(y)\|\to1.$   
		\end{enumerate}
	\end{enumerate}
\end{thm}
%
To finish this paper, combining Theorem \ref{thm_WO} with the main results in Section 6, we state the characterizations for the boundedness and the compactness of the operators   $ \widetilde{W}_{\psi,\varphi},  $  $ \widetilde{W}^0_{\psi,\varphi},  $ $ \widetilde{W}^{0,0}_{\psi,\varphi}. $

\begin{thm}
	Let   $W \subset Y'$ be a  separating  subspace. Let $\psi \in H(B_X)$ and $\varphi   \in S(B_X).$   Let  $ \mu, \nu $ be   normal weights on $ B_X. $ Then 
	\begin{enumerate}[\rm (1)]
		\item The  following are equivalent:
		\begin{itemize}
			\item $\widetilde{W}_{\psi,\varphi} $ is bounded;
				\item $\widetilde{W}^0_{\psi,\varphi} $ is bounded;
				\item  One of the assertions (1)-(4) in Theorem \ref{thm6_1},
		\end{itemize}
		\item  $\widetilde{W}^{0,0}_{\psi,\varphi}, $ is bounded if and only if the assumptions (1)-(4) in Theorem \ref{thm2_2} holds.
		\end{enumerate}
	\end{thm}

\begin{thm}
	Let   $W \subset Y'$ be a  almost norming subspace. Let $\psi \in H(B_X),$ $\varphi   \in S(B_X)$ and $ \mu, \nu $ be   normal weights on $ B_X. $ Assume that one of the following is satisfied:
\begin{enumerate}[\rm(a)]
	\item $ \int_0^1\frac{dt}{\nu(t)} = \infty $ and  the assertion (1) in Theorem \ref{thm_compact};
	\item $ \int_0^1\frac{dt}{\nu(t)} = \infty $ and  one of the assertions (4)-(6) in Theorem \ref{thm_compact} with  the additional assumption that (\ref{phi_relative_compact}) holds for some $ m\ge2; $
	\item $ \int_0^1\frac{dt}{\nu(t)} < \infty $ and  one of the assertions (1)-(3) in Theorem \ref{thm_compact2} with the additional assumption that (\ref{phi_relative_compact}) holds for some $ m\ge2.$
\end{enumerate}
Then  the  following are equivalent:
\begin{enumerate}[\rm (1)]
	\item $\widetilde{W}_{\psi,\varphi} $ is (resp. weakly) compact;
	\item $\widetilde{W}^0_{\psi,\varphi} $ is (resp. weakly) compact;
	\item The identity map $ I_Y: Y \to Y $ is   (resp. weakly) compact.
	\end{enumerate}
	\end{thm}
\begin{thm}
	Let   $W \subset Y'$ be a  almost norming subspace. Let $\psi \in H(B_X)$ and $\varphi   \in S(B_X).$ 
Let  $ \mu, \nu $ be   normal weights on $ B_X. $ Assume that one of the following is satisfied:
\begin{enumerate}[\rm(a)]
	\item The assertion (2) in Theorem \ref{thm2_2_cp};
	\item One of   the assertions (3)-(5) in Theorem \ref{thm2_2_cp} with the additional assumption that (\ref{phi_relative_compact}) holds for some $ m\ge2.$
\end{enumerate}
Then  the  following are equivalent:
\begin{enumerate}[\rm (1)]
	\item   $\widetilde{W}^{0,0}_{\psi,\varphi} $ is (resp. weakly) compact; 
	\item  The identity map $ I_Y: Y \to Y $ is (resp. weakly) compact.
\end{enumerate}
	\end{thm}


\end{document}